\documentclass[11pt,twoside]{amsart}
\usepackage{amsxtra}
\usepackage{color}
\usepackage{amsopn}
\usepackage{amsmath,amsthm,amssymb,arydshln}
\usepackage{mathrsfs,mathtools}
\usepackage{stmaryrd}
\usepackage{hyperref}
\usepackage{enumerate}
\usepackage{xcolor}
\usepackage{float}
\usepackage{pifont}
\usepackage{adjustbox}

\newtheorem{theorem}{Theorem}[section]
\newtheorem{corollary}[theorem]{Corollary}
\newtheorem{proposition}[theorem]{Proposition}
\newtheorem{lemma}[theorem]{Lemma}

\theoremstyle{definition}
\newtheorem{definition}[theorem]{Definition}

\newtheorem{example}[theorem]{Example}
\newtheorem{remark}[theorem]{Remark}

\makeatletter
\@namedef{subjclassname@2020}{%
  \textup{2020} Mathematics Subject Classification}
\makeatother

\newcommand{\R}{{\mathbb R}}

\newcommand{\W}{\wedge}

\newcommand{\f}{\varphi}

\newcommand{\psip}{\psi_{\scriptscriptstyle +}}
\newcommand{\psim}{\psi_{\scriptscriptstyle -}}
\newcommand{\Wp}{w_2^{\scriptscriptstyle +}}
\newcommand{\Wm}{w_2^{\scriptscriptstyle -}}
\newcommand{\Wu}{w_1}
\newcommand{\vol}{\mathrm{vol}}
\newcommand{\tomega}{\widetilde{\omega}}
\newcommand{\omegaz}{\omega_{\sst0}}

\newcommand{\ad}{{\mathrm{ad}}}

\newcommand{\SU}{\mathrm{SU}}

\newcommand{\G}{\mathrm{G}}

\newcommand{\fra}{\mathfrak{a}}
\newcommand{\frb}{\mathfrak{b}}

\newcommand{\frg}{\mathfrak{g}}

\newcommand{\frh}{\mathfrak{h}}
\newcommand{\frk}{\mathfrak{k}}
\newcommand{\frn}{\mathfrak{n}}

\newcommand{\fre}{\mathfrak{e}}

\newcommand{\frs}{\mathfrak{s}}
\newcommand{\frz}{\mathfrak{z}}

\newcommand{\fru}{\mathfrak{u}}

\newcommand{\st}{\ |\ }

\newcommand{\diag}{\mathrm{diag}}
\newcommand{\Der}{\mathrm{Der}}
\newcommand{\End}{\mathrm{End}}
\newcommand{\sst}{\scriptscriptstyle}
\renewcommand{\span}{\mathrm{span}}
\newcommand{\tr}{\mathrm{tr}}

\numberwithin{equation}{section}

\title[Closed G$_2$-structures on  unimodular Lie algebras with non-trivial center]{Closed G$_{\mathbf2}$-structures on unimodular  Lie algebras with non-trivial center}

\address{Dipartimento di Matematica ``G. Peano'' \\ Universit\`a degli Studi di Torino\\ Via Carlo Alberto 10\\10123 Torino\\ Italy}

\author{Anna Fino}\email{annamaria.fino@unito.it}
\author{Alberto Raffero}\email{alberto.raffero@unito.it}
\author{Francesca Salvatore}\email{francesca.salvatore@unito.it}


\begin{document}
\maketitle

\begin{abstract}
We characterize the structure of a seven-dimensional Lie algebra with non-trivial center endowed with a closed G$_2$-structure. 
Using this result, we classify all unimodular Lie algebras with non-trivial center admitting closed G$_2$-structures, up to isomorphism, 
and we show that six of them arise as the contactization of a symplectic Lie algebra.  
Finally, we prove that every semi-algebraic soliton on the contactization of a symplectic Lie algebra must be expanding, 
and we determine all unimodular Lie algebras with center of dimension at least two that admit semi-algebraic solitons, up to isomorphism. 
\end{abstract}

%%%%%%%%%%%%%%%%%%%%%%%%%%%%%%%%%%%%%%%%%%%%%%%%%%%%%%%%%%%%%%%%%%%%%%%%%%%%%%%%%%%%%%%%%
%%%%%%%%%%%%%%%%%%%%%%%%%%%%%%%%%%%%%%%%%%%%%%%%%%%%%%%%%%%%%%%%%%%%%%%%%%%%%%%%%%%%%%%%%
%																INTRODUCTION 
%%%%%%%%%%%%%%%%%%%%%%%%%%%%%%%%%%%%%%%%%%%%%%%%%%%%%%%%%%%%%%%%%%%%%%%%%%%%%%%%%%%%%%%%%
%%%%%%%%%%%%%%%%%%%%%%%%%%%%%%%%%%%%%%%%%%%%%%%%%%%%%%%%%%%%%%%%%%%%%%%%%%%%%%%%%%%%%%%%%

\section{Introduction}

A  $\G_2$-structure on a  seven-dimensional manifold $M$ is a reduction of the structure group of its frame bundle to the exceptional Lie group $\G_2$. 
This reduction exists if and only if $M$ is orientable and spin \cite{Gray}, and it is characterized by the existence of a $3$-form $\f\in\Omega^3(M)$ whose stabilizer at each point of $M$ is isomorphic to $\G_2$.  
Every such 3-form $\varphi$ induces a Riemannian metric $g_{\varphi}$ and an orientation on $M,$ and thus a Hodge star operator $*_\f$. 

\smallskip

Since every $7$-manifold admitting $\G_2$-structures is spin, it also admits almost contact structures. 
The interplay between these structures and the existence of contact structures on $7$-manifolds endowed with special types of $\G_2$-structures have been recently investigated in \cite{ACS, DLM}.

It is possible to construct examples of compact 7-manifolds with both a G$_2$-structure and a contact structure as follows. 
In \cite{BW}, Boothby and Wang showed that an even-dimensional compact manifold $N$ endowed with a symplectic form $\omegaz$ with integral periods    
is the base of a principal $\mathbb{S}^1$-bundle $\pi:M \rightarrow N$ having a connection 1-form $\theta$ that defines a contact structure on $M$ and satisfies the structure equation $d\theta = \pi^*\omegaz$. 
If $N$ is six-dimensional and it also admits a definite 3-form $\rho$ and a non-degenerate 2-form $\widetilde{\omega}$ which is taming for the almost complex structure $J$ induced by $\rho$ 
and one of the two orientations of $N,$ then the total space $M$ has a natural $\G_2$-structure defined by the 3-form 
$\f = \pi^*\widetilde{\omega} \W \theta +\pi^*\rho$ (see Sect.~\ref{IntroDef} for the relevant definitions). 
A special case of this construction occurs when the pair $(\tomega,\rho)$ defines an SU(3)-structure on $N,$ namely when the additional conditions $\tomega\W\rho=0$ and $3\,\rho\W J\rho = 2\,\tomega^3$ hold. 
On the other hand, a G$_2$-structure which is invariant under a free $\mathbb{S}^1$-action on a 7-manifold $M$ induces an SU(3)-structure on the orbit space $M/\mathbb{S}^1$ (see \cite{ApSa}).  

\smallskip

A $\G_2$-structure is said to be {\em closed} if the defining 3-form  $\varphi$ satisfies the condition $d \varphi  = 0$. 
If a closed G$_2$-structure $\varphi$ is also coclosed, i.e., $d*_\f\f=0$, then the induced metric $g_{\varphi}$ is Ricci-flat, and the Riemannian holonomy group ${\mathrm{Hol}} (g_{\varphi})$ 
is isomorphic to a subgroup of $\G_2$.  In this last case, the $\G_2$-structure is called {\em torsion-free}. 

Examples of 7-manifolds with a closed G$_2$-structure can be obtained taking products of lower dimensional manifolds endowed with suitable geometric structures 
(cf.~e.g.~\cite{ChSa,FeGr,FiYa,FiRa,LaLo}). 
Furthermore, the G$_2$-structure $\f = \pi^*\widetilde{\omega} \W \theta +\pi^*\rho$ on the total space of the $\mathbb{S}^1$-bundle $\pi:M\rightarrow N$  
is closed whenever the 2-form $\tomega$ on $N$ is symplectic and $\rho$ satisfies the condition $d\rho = -\omegaz \W {\tomega}$.  

In the literature, many examples of closed (non torsion-free) G$_2$-structures have been obtained on compact quotients of simply connected Lie groups by co-compact discrete subgroups (lattices). 
The first one was given by Fern\'andez on the compact quotient of a nilpotent Lie group \cite{Fer}.   
In the solvable non-nilpotent case, various examples are currently known, and lots of them satisfy additional meaningful conditions that one can impose on a closed G$_2$-structure, 
see e.g.~\cite{Fer1, FiRa, Fre, Lauret, Lau, LaNi}. 
In these examples, the closed G$_2$-structure on the Lie group $\G$ is left-invariant, and thus it is determined by a G$_2$-structure $\f$ on the Lie algebra $\frg = \mathrm{Lie}(\G)$ which is closed with respect to the 
Chevalley-Eilenberg differential of $\frg$.  Recall that a simply connected Lie group $\G$ admits lattices only if $\frg$ is unimodular \cite{Milnor}. 

The isomorphism classes of nilpotent and unimodular non-solvable Lie algebras admitting closed G$_2$-structures were determined in \cite{CoFe} and \cite{FiRa1}, respectively. 
So far, an analogous classification result for solvable non-nilpotent Lie algebras is missing. In fact, there is a lack of classification results for seven-dimensional solvable non-nilpotent Lie algebras. 

In this paper, we focus on seven-dimensional unimodular Lie algebras with non-trivial center admitting closed G$_2$-structures. 
It follows from \cite{FiRa1} that every Lie algebra of this type must be solvable. 

In Sect.~\ref{ClosedG2CExt}, we characterize the structure of a seven-dimensional Lie algebra $\frg$ with non-trivial center endowed with a closed G$_2$-structure $\f$. 
In detail, we show that it is the central extension of a six-dimensional Lie algebra $\frh$ by means of a closed 2-form $\omegaz\in\Lambda^2\frh^*$, 
and that $\f = \tomega\W \theta +\rho,$ where $\theta$ is a 1-form on $\frg$ satisfying $d\theta=\omegaz$, 
$\rho$ is a definite 3-form on $\frh$ such that $d\rho = -\omegaz\W\tomega$, and $\tomega$ is a symplectic form on $\frh$ that tames the almost complex structure induced by $\rho$ and a suitable orientation. 
If the 2-form $\omegaz$ is symplectic, the 1-form $\theta$ is a contact form on $\frg$ and $(\frg,\theta)$ is the {\em contactization} of $(\frh,\omegaz)$, see \cite{Alek}. 
In this last case, the Lie algebra $\frg$ admits both a closed G$_2$-structure and a contact structure.  
This is reminiscent of the construction involving the result of \cite{BW} that we sketched above. 

As a first consequence of this characterization, we determine all isomorphism classes of nilpotent Lie algebras admitting closed G$_2$-structures that arise as the contactization of a six-dimensional nilpotent 
Lie algebra (see Corollary \ref{CorNil}). 
Furthermore, in Sect.~\ref{ClassSect}, we use the characterization to classify all unimodular Lie algebras with non-trivial center admitting closed G$_2$-structures, up to isomorphism (see Theorem \ref{classification}). 
In addition to the nilpotent ones considered in \cite{CoFe}, we show that there exist eleven non-isomorphic solvable non-nilpotent Lie algebras satisfying the required conditions.  
Among them, only two can be obtained as the contactization of a six-dimensional symplectic Lie algebra. 
The simply connected Lie groups corresponding to some of these Lie algebras admit lattices. We use the results of \cite{Bock} to construct a lattice for two of them (see Remark \ref{latticesconstr}).    
In this way, we obtain new locally homogeneous examples of compact 7-manifolds with a closed G$_2$-structure. 
Finally, as a corollary of the classification result, we show that the abelian Lie algebra and a certain 2-step solvable Lie algebra are the only unimodular Lie algebras with non-trivial center 
admitting torsion-free G$_2$-structures. 

\smallskip

A special class of closed G$_2$-structures that has attracted a lot of attention in recent years is given by the Laplacian solitons. 
A closed G$_2$-structure $\f$ on a 7-manifold $M$ is said to be a {\em Laplacian soliton} if it satisfies the equation 
\begin{equation}\label{LapSolIntro}
\Delta_\f \f = \lambda \f+\mathcal{L}_{X}\f,
\end{equation}
for some real constant $\lambda$ and some vector field $X$ on $M,$ where $\Delta_\f$ denotes the Hodge Laplacian of the metric $g_\f$. 
These G$_2$-structures give rise to self-similar solutions of the G$_2$-Laplacian flow introduced by Bryant in \cite{Bry}, and they are expected to model finite time singularities of this flow 
(see \cite{Lot} for an account of recent developments on this topic). 

Depending on the sign of $\lambda$, a Laplacian soliton is called {\em expanding} ($\lambda>0$), {\em steady} ($\lambda=0$), or {\em shrinking} ($\lambda<0$).  
On a compact manifold, every Laplacian soliton which is not torsion-free must be expanding and satisfy \eqref{LapSolIntro} with $\mathcal{L}_{X} \varphi \neq 0$, see \cite{Lin, LoWe}.  
The existence of non-torsion-free Laplacian solitons on compact manifolds is still an open problem. 

In the non-compact setting, examples of Laplacian solitons of any type are known, see e.g.~\cite{Ball, FiRa, FiRa2, Fow, Lauret, Lau, LaNi, Nicolini}.  
In particular, the steady solitons in \cite{Ball} and the shrinking soliton in \cite{Fow} are inhomogeneous and of {\em gradient type}, i.e., $X$ is a gradient vector field. 
As for the known homogeneous examples, they consist of simply connected Lie groups $\G$ endowed with a left-invariant closed G$_2$-structure satisfying the equation \eqref{LapSolIntro} 
with respect to a vector field $X$ defined by a one-parameter group of automorphisms induced by a derivation $D$ of the Lie algebra $\frg = \mathrm{Lie}(\G)$. 
According to \cite{Lauret}, these solitons are called {\em semi-algebraic}. 

In Sect.~\ref{SolitonsSect}, we consider semi-algebraic solitons on unimodular Lie algebras with non-trivial center. 
Under a natural assumption on the derivation $D$, we are able to relate the constant $\lambda$ in \eqref{LapSolIntro} to a certain eigenvalue of $D$ and to the norm of the intrinsic torsion form of 
the semi-algebraic soliton $\f$, namely the unique 2-form $\tau$ such that $d*_\f\f = \tau\W\f = -*_\f\tau$. 
Moreover, we show that $\lambda$ coincides with the squared norm of $\tau$ whenever the Lie algebra is the contactization of a symplectic one (see Corollary \ref{lambdacont}). 
In this last case, the semi-algebraic soliton must be expanding. 
We also prove the non-existence of semi-algebraic solitons on certain Lie algebras with one-dimensional center, and we obtain the classification of all unimodular Lie algebras with center of 
dimension at least two that admit semi-algebraic solitons, up to isomorphism (see Theorem \ref{SASClassZ2}).

%%%%%%%%%%%%%%%%%%%%%%%%%%%%%%%%%%%%%%%%%%%%%%%%%%%%%%%%%%%%%%%%%%%%%%%%%%%%%%%%%%%%%%%%%
%%%%%%%%%%%%%%%%%%%%%%%%%%%%%%%%%%%%%%%%%%%%%%%%%%%%%%%%%%%%%%%%%%%%%%%%%%%%%%%%%%%%%%%%%
%																PRELIMINARIES 
%%%%%%%%%%%%%%%%%%%%%%%%%%%%%%%%%%%%%%%%%%%%%%%%%%%%%%%%%%%%%%%%%%%%%%%%%%%%%%%%%%%%%%%%%
%%%%%%%%%%%%%%%%%%%%%%%%%%%%%%%%%%%%%%%%%%%%%%%%%%%%%%%%%%%%%%%%%%%%%%%%%%%%%%%%%%%%%%%%%

\section{Preliminaries}\label{SectPrel}

\subsection{Definite forms and geometric structures in six and seven dimensions}\label{IntroDef}
In this section, we review the properties of the geometric structures defined by a differential form satisfying one of the following conditions. 
\begin{definition} \ 
\begin{enumerate}[-]
\item a 3-form $\rho$ on a six-dimensional vector space $V$ is said to be {\em definite} if for each non-zero vector $v\in V$ the contraction $\iota_v\rho$ has rank four;
\item a 3-form $\varphi$ on a seven-dimensional vector space $W$ is said to be {\em definite} if for each non-zero vector $w\in W$ the contraction $\iota_w\varphi$ has rank six. 
\end{enumerate}
\end{definition}

Let $V$ be a six-dimensional vector space, choose an orientation $\Omega\in\Lambda^6V^*$ on it, and let $\rho\in\Lambda^3V^*$ be a definite 3-form. 
Then, the pair $(\rho,\Omega)$ gives rise to a complex structure $J=J_{\rho,\Omega}$ on $V$ as follows. 
Consider the endomorphism $S_\rho:V\rightarrow V$ defined via the identity
\[
\iota_v\rho\wedge\rho\wedge \eta = \eta(S_\rho(v))\Omega, 
\]
for all $\eta\in V^*$. Then $S_\rho^2=P(\rho) \mathrm{Id}_V$ for some quartic polynomial $P(\rho)$, and $\rho$ is definite precisely when $P(\rho)<0$. 
The complex structure $J=J_{\rho,\Omega}$ is given by $J\coloneqq (-P(\rho))^{-1/2}\,S_\rho$.

If $\omega$ is a non-degenerate 2-form on $V$ such that $\omega\W \rho=0$, and $J$ is the complex structure induced by $\rho$ and the orientation $\omega^3$, 
then the bilinear form $g\coloneqq \omega(\cdot,J\cdot)$ is symmetric and non-degenerate. 
In this case, the pair $(\omega,\rho)$ defines an SU(3)-structure if and only if $g$ is positive-definite and $3\,\rho\W J\rho = 2\,\omega^3$.  

In what follows, we shall denote an SU(3)-structure by the usual notation $(\omega,\psip)$. Moreover, we denote by $\psim\coloneqq J\psip = \psip(J\cdot,J\cdot,J\cdot)$ the imaginary part 
of the complex $(3,0)$-form $\psip + i\psim$, and by $\vol_g = \frac16\omega^3$ the volume form corresponding to the metric $g$.  
Finally, the Hodge operator induced by $g$ and the orientation $\vol_g$ will be denoted by $*_g$. 

The following identities will be useful in the sequel. The reader may refer to \cite[Lemma 3.7]{FiRa} for a proof. 
\begin{lemma}\label{su3identities}
Let $(\omega,\psip)$ be an $\SU(3)$-structure on a six-dimensional vector space $V,$ and let $\alpha\in V^*$. Then,
\begin{enumerate}[{\rm i)}]
\item\label{idi} $*_g(\alpha\W\psim)\W\omega =J\alpha\W\psip = \alpha\W\psim$;
\item\label{idii} $*_g(\alpha\W\psim)\W\omega^2 = 0$;
\item\label{idiii} $*_g(\alpha\W\psim)\W\psip = -*_g(\alpha\W\psip)\W\psim =  \alpha\W\omega^2=2*_g(J\alpha)$;
\item\label{idiv} $*_g(\alpha\W\psim)\W\psim = *_g(\alpha\W\psip)\W\psip  = -J\alpha\W\omega^2=2*_g\alpha$.
\end{enumerate}
\end{lemma}

\smallskip
    
A definite 3-form $\varphi$ on a seven-dimensional vector space $W$ defines a G$_2$-structure on it. In detail, $\varphi$ gives rise to a symmetric bilinear map 
\begin{equation}\label{bphi}
b_\varphi : W \times W \rightarrow \Lambda^7W^*, \quad b_\varphi(v,w) = \frac16\, \iota_v\varphi \wedge \iota_w \varphi \wedge \varphi, 
\end{equation}
the 7-form $\det(b_\varphi)^{1/9}$ is different from zero, $g_\varphi\coloneqq \det(b_\varphi)^{-1/9} \, b_\varphi$ is an inner product on $W,$ 
and $\mathrm{vol}_{g_\f} = \det(b_\varphi)^{1/9}$ is the corresponding volume form. The Hodge operator induced by $g_\f$ and $\vol_{g_\f}$ will be denoted by $*_\f$. 

\medskip
Now, consider a seven-dimensional vector space $W,$ and let $\varphi$ be a 3-form on it. 
Choose a non-zero vector $z\in W$ and a complementary subspace $V\subset W$ so that $W \cong V \oplus \mathbb{R}z$. Then, we can write 
\[
\varphi = \tomega \wedge \theta + \rho, 
\]
where $\theta\in W^*$ is the dual of $z$, and $\tomega\in\Lambda^2V^*$, $\rho\in\Lambda^3V^*$. 
The 3-form $\varphi$ on $W$ is definite if and only if the 3-form $\rho$ on $V$ is definite and $\tomega$ is a taming form for the complex structure $J$ induced by $\rho$ 
and one of the two orientations of $V$, namely $\tomega(v,Jv)>0$ for every non-zero vector $v\in V$.

\begin{remark}\label{RemDef}
If $\f= \tomega \wedge \theta + \rho$ is definite, then the 2-form $\tomega=\iota_z\f|_V$ on $V$ has rank six, namely it is a non-degenerate 2-form. 
The pair $(\tomega,\rho)$ defines an SU(3)-structure on $V,$ up to a suitable normalization, if and only if $\tomega\W\rho=0$. 
When this happens, the vector space $V$ coincides with the $g_\f$-orthogonal complement of $\R z\subset W.$
\end{remark}

On the other hand, if $\f$ defines a G$_2$-structure on $W,$ we can consider the six-dimensional subspace $U\coloneqq (\R z)^\perp\subset W$ 
and the $g_\f$-orthogonal splitting $W = U \oplus \R z$. 
Then, if we let $u\coloneqq |z|_\f = g_\f(z,z)^{1/2}$ and $\eta\coloneqq u^{-2} z^\flat$, so that $\eta(z)=1$, we have
\[
\f = u\,\omega\W\eta+\psip,\quad *_\f\f = \frac{1}{2}\,\omega\W\omega + u\,\psim\W\eta, \quad g_\f = g + u^2\eta\otimes\eta,  
\]
and the pair $(\omega,\psip)$ defines an SU(3)-structure on $U$ inducing the metric $g$. In particular, $\mathrm{vol}_{g_\f} = u\,\mathrm{vol}_g \W \eta$. 
Since the vector subspaces $V$ and $U$ are isomorphic, there exists an SU(3)-structure on $V$ corresponding to $(\omega,\psip)$ via the identification $V\cong U$. 
We shall denote this SU(3)-structure using the same symbols. 
It follows from the discussion in Remark \eqref{RemDef} that $V$ and $U$ coincide if and only if $\tomega\W\rho=0.$ In such a case, $\eta$ and $\theta$ coincide, too.

\begin{remark}\label{remcontsu3}\ 
\begin{enumerate}[1)]
\item The structures $(\tomega,\rho)$ and $(\omega,\psip)$ on $V$ are related as follows. 
On $W = V \oplus \R z$, we have $\f= \tomega \wedge \theta + \rho$ and $\f= u\,\omega\W\eta+\psip$. Thus, $\tomega = \iota_z\f = u\,\omega$. 
Moreover, since $\eta(z)=1$, we can consider the decomposition $\eta=\eta_V+\theta$, where $\eta_V \in V^*$, and see that 
\[
\rho = u\,\omega\W\eta_V + \psip.
\]
\item Let $\{e_1,\ldots,e_6,e_7\}$ be a basis of $W = V\oplus \R z$ with $V = \span\{e_1,\ldots,e_6\}$ and $e_7=z$, and denote by $\{e^1,\ldots,e^6\}$ the dual basis of $V^*$. 
Then, a basis of $U=(\R e_7)^\perp$ is given by $\left\{e_k -\frac{g_\f(e_k,e_7)}{u^2} e_7, k=1,\ldots,6 \right\}$. Consequently, $\{e^1,\ldots,e^6\}$ is a basis of $U^*$.  
\end{enumerate}
\end{remark}

\subsection{SU(3)- and G$_{\mathbf2}$-structures on Lie algebras}\label{SU3G2LieAlg} 
We now focus on SU(3)- and G$_2$-structures on Lie algebras. 
It is well-known that an SU(3)-structure on a six-dimensional Lie algebra $\frh$ gives rise to a left-invariant SU(3)-structure on every Lie group corresponding to $\frh$. 
Conversely, a left-invariant SU(3)-structure $(\omega,\psip)$ on a six-dimensional Lie group $\mathrm{H}$ is determined by the SU(3)-structure $\left.(\omega,\psip)\right|_{1_{\mathrm{H}}}$ 
on $T_{1_{\mathrm{H}}}\mathrm{H}\cong \frh$. An analogous correspondence holds for $\G_2$-structures on seven-dimensional Lie algebras and left-invariant $\G_2$-structures on 
seven-dimensional Lie groups. 

Consider a six-dimensional Lie algebra $\frh$ endowed with an SU$(3)$-structure $(\omega,\psip)$. 
The natural action of SU$(3)$ on the space of $k$-forms $\Lambda^k\frh^*$, for $k=2,3$, gives rise to the following splittings (cf.~\cite{BeVe,ChSa}) 
\[
\Lambda^2\frh^* =  \mathbb{R}\omega\oplus\Lambda^2_6\frh^*\oplus\Lambda^2_8\frh^*, \quad  
\Lambda^3 \frh^*=\mathbb{R}\psip \oplus\mathbb{R}\psim \oplus\Lambda^3_6\frh^*\oplus\Lambda^3_{12}\frh^*,
\]  
where the irreducible $r$-dimensional SU(3)-modules $\Lambda^k_r\frh^*$ are defined as follows
\[
\Lambda^2_6\frh^*  = \left\{ \sigma \in \Lambda^2 \frh^* \, \lvert \, J\sigma=-\sigma \right\}, \quad
\Lambda^2_8 \frh^* = \left\{ \beta \in \Lambda^2 \frh^* \, \lvert \, J\beta=\beta, ~ \beta\wedge \omega^2=0 \right\},
\]
and
\[
\Lambda^3_6\frh^* = \left\{ \alpha\wedge \omega \, \lvert \, \alpha\in \frh^* \right\}, \quad
\Lambda^3_{12}\frh^* = \left\{\gamma \in \Lambda^3 \frh^* \, \lvert \, \gamma\wedge \omega=0,  ~ \gamma\wedge \psi_{\sst\pm}=0 \right\}.
\]
The SU(3)-irreducible decomposition of the space $\Lambda^4\frh^*$ follows from that of $\Lambda^2\frh^*$ via the identity $\Lambda^4\frh^* = *_g \left( \Lambda^2\frh^*\right)$. 
Notice that every 2-form $\sigma \in \Lambda^2_6\frh^*$ satisfies the identity $\sigma \wedge \omega = *_g \sigma$, 
and every 2-form $\beta \in \Lambda^2_8\frh^* $ satisfies the identity $\beta \wedge \omega = -*_g\beta$.

Let $d$ denote the Chevalley-Eilenberg differential of $\frh$. 
In view of the previous decompositions, there exist unique $w_0^+, w_0^- \in\R$, $\nu_1,w_1\in\frh^*$,  $w_2^+, w_2^- \in \Lambda^2_8\frh^*$, $w_3\in\Lambda^3_{12}\frh^*$, such that 
\[
\begin{split}
d\omega 	&= -\frac32w_0^-\psip +\frac32 w_0^+\psim + w_3 + \nu_1\W\omega,\\
d\psip 	&=w_0^+\omega^2 +w_2^+\W\omega + w_1\W\psip,\\
d\psim 	&= w_0^-\omega^2 +w_2^-\W\omega +Jw_1\W\psip,
\end{split}
\]
By \cite{BeVe,ChSa}, the intrinsic torsion of an SU(3)-structure $(\omega,\psip)$ is completely determined by these forms, which are thus called the {\em intrinsic torsion forms} of $(\omega,\psip)$. 

A similar result holds for G$_2$-structures, see \cite{Bry,FeGr}. Here, we shall focus on seven-dimensional Lie algebras $\frg$ endowed with a {\em closed} G$_2$-structure $\f$. 
In such a case, $d\f=0$, $d$ being the Chevalley-Eilenberg differential of $\frg$, 
and the intrinsic torsion of $\f$ is determined by the unique 2-form $\tau \in \Lambda^2_{14}\frg^*\coloneqq \{\alpha \in \Lambda^2 \frg^* \, \lvert \, \alpha \wedge \f =  -*_\f\alpha \}$ such that 
\[
d*_{\varphi}\varphi=\tau\wedge \varphi.
\]
In particular, a closed G$_2$-structure is torsion-free if and only if $\tau=0$. 
We shall refer to $\tau$ as the {\em intrinsic torsion form} of the closed G$_2$-structure $\f$.

\subsection{The contactization of a symplectic Lie algebra}\label{SectCont} 
Let $\frh$ be a Lie algebra of dimension $n\geq2$ and denote by $[\cdot , \cdot]_{\frh}$ its Lie bracket.   
Consider a 2-form $\omegaz\in\Lambda^2\frh^*$ that is closed with respect to the Chevalley-Eilenberg differential $d_\frh$ of $\frh$.   
The {\em central extension} of $(\frh,\omegaz)$ is the real Lie algebra of dimension $n+1$ defined as follows: consider the vector space 
\[
\frg \coloneqq \frh \oplus \R z, 
\]
and endow it with the Lie bracket
\begin{equation}\label{BracketCE}
 [z,\frh]=0, \qquad [x,y] = - \omegaz (x,y)z+ [x,y]_{\frh},~\forall x,y \in \frak h. 
\end{equation}
It is clear from this definition that the vector $z$ belongs to the center of $\frg$. 
More precisely, the center of $\frg$ is given by
\[
\frz(\frg) = \left(\frz(\frh) \cap \mathrm{Rad}(\omegaz) \right) \oplus \R z,
\]
where $ \mathrm{Rad}(\omegaz) \coloneqq \left\{x \in \mathfrak{h} \st \omegaz(x,y)=0,~\forall\,  y \in \mathfrak{h}\right\}$. 
Note that the central extension of $(\frh,\omegaz)$ only depends on the cohomology class $[\omegaz]\in H^2(\frh)$. Indeed, different representatives of $[\omegaz]$ give rise to isomorphic 
central extensions.  

In the following, we shall denote by $\theta\in\frg^*$ the dual covector of $z$ uniquely defined by the conditions $\theta (z) = 1$ and $\theta|_\frh=0$. 

Let $d$ denote the Chevalley-Eilenberg differential of $\frg$. Then, $d \theta$ defines an exact $2$-form on $\frg$ that coincides with $\omegaz$ on $\frh$ and satisfies $\iota_z d\theta=0$. 
Thus, we can write $d\theta = \omegaz$ on $\frg$ by extending $\omegaz$ to $\frg$ via the condition $\iota_z \omegaz=0$. 
When $\omegaz$ is zero, the Lie algebra $\frg$ is just the direct sum of $\frh$ and the abelian Lie algebra $\R$. 

\smallskip

When $(\frh, \omegaz)$ is a symplectic Lie algebra of dimension $2n$, the previous construction gives rise to a contact Lie algebra $(\frg, \theta)$ of dimension $2n+1$ with center $\frz(\frg)=\R z$. 
Indeed, $\theta\in\frg^*$ satisfies the condition $(d\theta)^n\W\theta\neq0$, and so it is a contact form on $\frg$ (see \cite{Alek}). 
Moreover, $\mathrm{Rad}(\omegaz)=\{0\}$, whence $\frz(\frg)=\R z$. 
Notice that $z$ is the {\em Reeb vector} of the contact structure $\theta$, as $\theta(z)=1$. 
This motivates the following. 

\begin{definition} 
The contact Lie algebra $(\frg,\theta)$ obtained from the symplectic Lie algebra $(\frh,\omegaz)$ via the construction described above is called the {\em contactization} of $(\frh,\omegaz)$. 
\end{definition}

It is easy to characterize contact Lie algebras arising as the contactization of a symplectic Lie algebra, as the next result shows (see also \cite{Di}).
\begin{proposition}  
A contact Lie algebra $(\frg,\theta)$ is the contactization of a symplectic Lie algebra $(\frh,\omegaz)$ if and 
only if the center $\frz(\frg)$ of $\frg$ is not trivial.  
\end{proposition}
\begin{proof} 
If $(\frg,\theta)$ is the contactization of a symplectic Lie algebra $(\frh,\omegaz)$, then the assertion is true. 
Conversely, let us assume that $(\frg,\theta)$ is a contact Lie algebra of dimension $2n+1$ with non-trivial center. 
Then, $\frz(\frg)$ is one-dimensional and it is spanned by the Reeb vector $z$ of the contact structure $\theta$ (cf.~\cite[Prop.~1]{AFV}).   
Consequently, we can consider the decomposition $\frg = \frh\oplus\R z$, where the $2n$-dimensional subspace $\frh\coloneqq\ker \theta$ is a Lie algebra with respect to the bracket  
\[
[x,y]_\frh \coloneqq [x,y] - \theta\left([x,y]\right)z,\quad x,y\in\frh.
\]
Let $\omegaz$ be the 2-form on $\frh$ defined as $\omegaz(x,y) = d\theta(x,y)$, for all $x,y\in\frh$. 
A direct computation shows that $\omegaz$ is closed with respect to the Chevalley-Eilenberg differential $d_\frh$ of $\frh$. Moreover, $\omegaz$ is non-degenerate. 
Indeed, $(d\theta)^n\W\theta$ is a volume form on $\frg$ and contracting it with $z$ gives $(d\theta)^n\neq0$, as  $\theta(z)=1$ and $\iota_z d\theta  = -\theta([z,\cdot])  = 0$. 
Therefore, $(\frg,\theta)$ is the contactization of the symplectic Lie algebra $(\frh,\omegaz)$. 
\end{proof}

%%%%%%%%%%%%%%%%%%%%%%%%%%%%%%%%%%%%%%%%%%%%%%%%%%%%%%%%%%%%%%%%%%%%%%%%%%%%%%%%%%%%%%%%%
%%%%%%%%%%%%%%%%%%%%%%%%%%%%%%%%%%%%%%%%%%%%%%%%%%%%%%%%%%%%%%%%%%%%%%%%%%%%%%%%%%%%%%%%%
%																CLOSED G2 CENTRAL 
%%%%%%%%%%%%%%%%%%%%%%%%%%%%%%%%%%%%%%%%%%%%%%%%%%%%%%%%%%%%%%%%%%%%%%%%%%%%%%%%%%%%%%%%%
%%%%%%%%%%%%%%%%%%%%%%%%%%%%%%%%%%%%%%%%%%%%%%%%%%%%%%%%%%%%%%%%%%%%%%%%%%%%%%%%%%%%%%%%%

\section{Closed G$_2$-structures on central extensions of Lie algebras}\label{ClosedG2CExt}

In this section, we investigate the structure of a seven-dimensional Lie algebra with non-trivial center endowed with a closed G$_2$-structure. We begin with the following. 

\begin{proposition}\label{ClosedContactization} 
Let $\frh$ be a six-dimensional Lie algebra and let $\omegaz$ be a closed 2-form on it.  
Assume that $\frh$ admits a definite 3-form $\rho$ and a symplectic form $\tomega$ such that 
\begin{enumerate}[a)]
\item $\tomega$ is a taming form for the almost complex structure $J_\rho$ on $\mathfrak{h}$ induced by $\rho$ and one of the two orientations of $\frh$; 
\item $d  \rho = - \tomega \wedge \omegaz$. 
\end{enumerate}
Then, the seven-dimensional Lie algebra $\mathfrak{g} \coloneqq \frh \oplus\R z$ obtained as the central extension of  $(\mathfrak{h},\omegaz)$ is endowed with 
a closed $\G_2$-structure defined by the 3-form 
\[
\varphi =  \tomega  \wedge \theta + \rho. 
\]
\end{proposition}

\begin{proof}
The hypothesis on $\rho$ and $\tomega$ guarantee that the 3-form $\varphi =   \tomega  \wedge \theta + \rho$ defines a G$_2$-structure on $\frg = \frh \oplus \R z$. 
Moreover, since $\tomega$ is closed and $\omegaz=d\theta$, we have
\[
d \varphi = d \tomega \wedge \theta + \tomega\wedge d\theta  + d \rho = \tomega\wedge\omegaz+d\rho = 0. \vspace{-0.2cm}
\] 
\end{proof}

The next result is a converse of Proposition \ref{ClosedContactization}. 

\begin{proposition} \label{invClosedContactization}
Let $\frg$ be a seven-dimensional Lie algebra endowed with a closed $\G_2$-structure $\f$.  
Assume that the center of $\frg$ is not trivial, consider a non-zero vector $z\in\frz(\frg)$ and denote by $\theta\in\frg^*$ its dual $1$-form.    
Then, $\frg$ is the central extension of a six-dimensional Lie algebra $(\frh,\omegaz)$, 
and the closed $\G_2$-structure can be written as $\f =  \tomega  \wedge \theta + \rho$,  
where $\rho$ is a definite 3-form on $\frh$,  $\tomega$ is a taming symplectic form for $J_\rho$ and $d  \rho = - \tomega \wedge \omegaz$. 
\end{proposition}

\begin{proof} 
Consider the six-dimensional subspace $\frh\coloneqq \ker\theta$ of $\frg$. Then, $\frh$ is a Lie algebra with respect to the bracket 
\begin{equation}\label{bracketh}
[x,y]_\frh \coloneqq [x,y] - \theta\left([x,y]\right)z,\quad x,y\in\frh, 
\end{equation}
and the 2-form $\omegaz\coloneqq d\theta|_{\frh\times\frh}$ on $\frh$ is closed with respect to the Chevalley-Eilenberg differential of $(\frh,[\cdot,\cdot]_\frh)$. 
In particular, $\frg = \frh\oplus \R z$ is the central extension of  $(\frh,\omegaz)$. 
Since $\Lambda^3\frg^* = (\Lambda^2\frh^*\otimes \R \theta) \oplus \Lambda^3 \frh^*$,  
there exist some ${\tomega} \in \Lambda^2\frh^*$, $\rho\in\Lambda^3\frh^*$ such that $\f =  \tomega  \wedge \theta + \rho$. 
Clearly, $\rho$ is a definite 3-form on $\frh$ and $\tomega=\iota_z\f$ is a non-degenerate taming 2-form for $J_\rho$.  
Moreover, ${\tomega}$ is symplectic. Indeed, 
\[
0 = \mathcal{L}_z\f = d(\iota_z\f) = d{\tomega},
\]
as $z\in\frz(\frg)$. Finally, we have
\[
0  =d\f = {\tomega}\W d\theta + d\rho = {\tomega}\W \omegaz + d\rho, 
\]
where the last identity follows from $\iota_z d\theta=0$. 
\end{proof}

\begin{remark}\ 
\begin{enumerate}[1)]
\item It follows from \cite{FiRa1} that every seven-dimensional unimodular Lie algebra with non-trivial center admitting closed $\G_2$-structures is necessarily solvable. 
On the other hand, there exist unimodular solvable centerless Lie algebras admitting closed $\G_2$-structures, see e.g.~\cite[Ex.~4.7]{Lau}. 
\item Any vector $z\in\mathfrak{z}(\frg)$ satisfies $\mathcal{L}_z\f=0$. More generally, if $x\in\frg$ satisfies $\mathcal{L}_x\f=0$, 
then $\mathcal{L}_x g_\f=0$, whence it follows that $\ad_x\in\Der(\frg)$ is skew-symmetric. Consequently, if the Lie algebra $\frg$ is completely solvable, namely the spectrum of $\ad_v$ is real for all $v\in\frg$, 
then every vector $x$ satisfying  $\mathcal{L}_x\f=0$ must belong to the center of $\frg$.    
\end{enumerate}
\end{remark}

Before showing a first consequence of Proposition \ref{invClosedContactization}, we recall some notations that we will use from now on. 
Given a Lie algebra of dimension $n$, we write its structure equations with respect to a basis of covectors $\left(\varepsilon^1,\ldots,\varepsilon^n \right)$
by specifying the $n$-tuple $(d \varepsilon^1,\ldots, d \varepsilon^n)$. 
Moreover, we use the notation $\varepsilon^{ijk\cdots}$ as a shorthand for the wedge product of covectors $\varepsilon^i \W \varepsilon^j \W \varepsilon^k \W \cdots$.  

\begin{corollary}\label{CorNil}
Let $\frg$ be a seven-dimensional nilpotent Lie algebra endowed with a closed $\G_2$-structure $\f$.  
Then $\frg$ is the central extension of a six-dimensional nilpotent Lie algebra $\frh$ admitting symplectic structures. 
Moreover, $\frg$ is the contactization of a six-dimensional symplectic nilpotent Lie algebra  $(\frh, \omegaz)$ if and only if $\frg$ is isomorphic to one of the following:  
\[
\begin{split}
\frn_9	&=(0,0,e^{12},e^{13},e^{23},e^{15} + e^{24},e^{16} + e^{34} + e^{25}),\\
\frn_{10}	&=(0,0,e^{12},0,e^{13} + e^{24},e^{14},e^{46} + e^{34} + e^{15} + e^{23}),\\
\frn_{11}	&=(0,0,e^{12},0,e^{13},e^{24} + e^{23},e^{25} + e^{34} + e^{15} + e^{16} - 3 e^{26}), \\
\frn_{12}	&=(0, 0, 0, e^{12}, e^{23}, -e^{13}, 2 e^{26} - 2 e^{34} -2 e^{16} + 2 e^{25}).
\end{split}
\]
\end{corollary}

\begin{proof}  
Since a nilpotent Lie algebra has non-trivial center, the first assertion immediately follows from Proposition \ref {invClosedContactization}. 
By the classification result of \cite{CoFe}, a seven-dimensional nilpotent Lie algebra admitting closed $\G_2$-structures is isomorphic to one of the following:
\[
\begin{split}
&\frn_1 =(0,0,0,0,0,0,0),\\
&\frn_2 = (0,0,0,0,e^{12},e^{13},0),\\
&\frn_3 =(0,0,0,e^{12},e^{13},e^{23},0),\\
&\frn_4 = (0,0,e^{12},0,0,e^{13} + e^{24},e^{15}),\\
&\frn_5 = (0,0,e^{12},0,0,e^{13},e^{14} + e^{25}),\\
&\frn_6 =(0,0,0,e^{12},e^{13},e^{14},e^{15}),\\
&\frn_7 = (0, 0, 0, e^{12}, e^{13}, e^{14} + e^{23}, e^{15}),\\
&\frn_8 = (0, 0, e^{12}, e^{13}, e^{23}, e^{15} + e^{24}, e^{16} + e^{34}),\\
&\frn_9 = (0,0,e^{12},e^{13},e^{23},e^{15} + e^{24},e^{16} + e^{34} + e^{25}),\\
&\frn_{10} = (0,0,e^{12},0,e^{13} + e^{24},e^{14},e^{46} + e^{34} + e^{15} + e^{23}),\\
&\frn_{11} = (0,0,e^{12},0,e^{13},e^{24} + e^{23},e^{25} + e^{34} + e^{15} + e^{16} - 3 e^{26}), \\
&\frn_{12} = (0, 0, 0, e^{12}, e^{23}, -e^{13}, 2 e^{26} - 2 e^{34} -2 e^{16} + 2 e^{25}).
\end{split}
\]

By \cite[Thm.~4.2]{Ku}, a decomposable nilpotent Lie algebra cannot admit any contact structure. 
Consequently, the Lie algebras $\frn_1$, $\frn_2$ and $\frn_3$ cannot be the contactization of any symplectic Lie algebra. 
Seven-dimensional indecomposable nilpotent Lie algebras admitting contact structures have been classified in \cite{Ku}. 
Comparing this classification with the one above, we see that $\frg$ must be isomorphic to one of $\frn_9$, $\frn_{10}$, $\frn_{11}$, $\frn_{12}$. 
For each of these Lie algebras, $\frz(\frn_i) = \R e_7$ and the 2-form $de^7$ induces a symplectic form on the six-dimensional nilpotent Lie algebra $\frh_i \coloneqq \ker e^7$ 
with Lie bracket defined as in \eqref{bracketh}. 
\end{proof}

Let us now consider a seven-dimensional Lie algebra $\frg$ with non-trivial center endowed with a closed G$_2$-structure $\f$. 
Then, from the previous discussion we can assume that $\frg = \frh\oplus \R z$ is the central extension of a six-dimensional Lie algebra $(\frh,\omegaz\coloneqq d\theta|_{\frh\times\frh})$, 
and that $\f = \tomega\W\theta+\rho$, with $d\rho = -\tomega\W\omegaz$ and $d\tomega=0$. 
From Sect.~\ref{SectPrel}, we also know that $\frh$ admits an SU(3)-structure $(\omega,\psip)$ such that $\f = u\,\omega\W\eta+\psip,$ 
where $u\coloneqq |z|_\f$ and $\eta\coloneqq u^{-2} z^\flat = \eta_\frh+\theta$, for some $\eta_\frh\in\frh^*$. 
In particular, $\frh$ is the $g_\f$-orthogonal complement of $\R z$ in $\frg$ if and only if $\eta_\frh=0$. 
It is worth observing that this setting generalizes the one considered in \cite[Sect.~6.1]{FiRa}, which corresponds to the case when both $\eta_\frh=0$ and $\omegaz=0$,  
i.e., to the direct sum of Lie algebras $\frg = \frh \oplus \R$ endowed with a closed G$_2$-structure inducing the product metric.  

\smallskip

We now investigate the properties of the SU(3)-structure $(\omega,\psip)$ on $\frh$. Since $u\,\omega=\tomega$, we immediately see that $d\omega=0$. 
Consequently, we have
\begin{equation}\label{dpsiCE}
\begin{split}
d\psip	&= \Wp\W\omega + \Wu\W\psip,\\
d\psim	&= \Wm\W\omega + J\Wu\W\psip,
\end{split}
\end{equation}
for some unique $\Wu\in\frh^*$ and $w_2^{\pm}\in\Lambda^2_8\frh^*$ (cf.~Sect.~\ref{SU3G2LieAlg}). 

\begin{lemma}\label{lemdeta}
	The $2$-form $d\eta\in\Lambda^2\frh^*$ has no component in $\Lambda^2_1\frh^*=\R\omega,$ that is, $d\eta\W \omega^2=0$. 
	Moreover, the intrinsic torsion forms $\Wm$ and $\Wu$ are related to the components $(d\eta)_k$ of $d\eta$ in $\Lambda^2_k\frh^*,$ $k=6,8$, as follows 
	\[
	\begin{split}
		u\,(d\eta)_6 & = -*_g(\Wu \W \psip),\\
		u\,(d\eta)_8 & = -\Wp.
	\end{split}
	\]
	In particular, $w_1 = \frac{u}{2} *_g(\psip\W d\eta).$
\end{lemma}
\begin{proof}
	The condition $d\f=0$ is equivalent to $d\psip = -u\,\omega\W d\eta$. Since $\omega$ is symplectic and $\omega\W\psip=0$, we get $d\eta\W\omega^2=0$. 
	Now, according to the SU(3)-irreducible decomposition $\Lambda^2\frh^* = \Lambda^2_1\frh^*\oplus\Lambda^2_6\frh^*\oplus\Lambda^2_8\frh^*$, this implies that $(d\eta)_1=0$, and we thus have 
	$d\eta = (d\eta)_6+(d\eta)_8$, with $(d\eta)_6\W\omega = *_g(d\eta)_6$ and $(d\eta)_8\W\omega = -*_g(d\eta)_8$, see Section \ref{SU3G2LieAlg}. Therefore, 
	\[
	d\psip = -u\,\omega\W d\eta = -u\,*_g(d\eta)_6 + u *_g(d\eta)_8. 
	\]
	Comparing this expression with the one in \eqref{dpsiCE}, we obtain the identities relating $(d\eta)_6$ and $(d\eta)_8$ with $w_1$ and $\Wp$, respectively.  	
	Finally, to obtain the expression of $w_1$, it is sufficient to notice that
	\[
	u\,d\eta\W\psip = u\,(d\eta)_6\W \psip = -*_g(\Wu \W \psip)\W\psip = -2*_g w_1, 
	\]
	where the last identity follows from Lemma \ref{su3identities}.
\end{proof}

Since $d\eta\in\Lambda^2_6\frh^*\oplus\Lambda^2_8\frh^*$, it satisfies the following condition (see \cite[Rem.~2.7]{BeVe}): 
\begin{equation}\label{detaomega}
d\eta\W\omega = -J*_gd\eta. 
\end{equation}

In the next lemma we describe the intrinsic torsion form $\tau$ of the closed G$_2$-structure $\f= u\,\omega\W\eta+\psip$ on $\frg$ 
in terms of the torsion forms of the SU(3)-structure $(\omega,\psip)$ on $\frh$. 

\begin{lemma}\label{tauLemma}
The intrinsic torsion form $\tau\in\Lambda^2_{14}\frg^*$ of the closed $\G_2$-structure $\f= u\,\omega\W\eta+\psip$ is given by 
\[
\tau = \Wm - *_g\left(Jw_1\W\psip \right) -2u\,Jw_1\W\eta, 
\]
and its Hodge dual is 
\[
*_\f\tau = u*_g\Wm\W\eta - u\,J\Wu\W\psip\W\eta+2*_gJ\Wu. 
\]
Consequently, $|\tau|_\f^2  = |\Wm|_g^2+6|\Wu|_g^2$. 
\end{lemma}
\begin{proof}
Recall that $\tau = - *_\f d *_\f \f$, where $*_\f\f = \frac12\, \omega^2 +u\,\psim\W\eta$. We first compute
\[
\begin{split}
d*_\f\f	&= u\,d\psim\W\eta - u\,\psim\W d\eta = u\,d\psim\W\eta - u\,\psim\W (d\eta)_6\\	
		&= u\,d\psim\W\eta -2\,*_gJw_1,
\end{split}
\]
where we used $(d\eta)_8\W\psim=0$, Lemma \ref{lemdeta} and the identity $*_g(w_1\W\psip)\W\psim = -2*_gJw_1$. 

Using now the relation between the Hodge operators $*_\f$ and $*_g$ together with the identity $d\psim = -*_g\Wm +Jw_1\W\psip$, we obtain
\[
\tau = -*_\f d *_\f\f	 = -*_gd\psim -2\,u\,Jw_1\W\eta =  \Wm - *_g\left(Jw_1\W\psip \right) -2u\,Jw_1\W\eta. 
\]
From this expression and the relation between $*_\f$ and $*_g$, one can easily compute $*_\f\tau$. To obtain the norm of $\tau$, it is then sufficient to 
use the identity $\tau\W*_\f\tau = |\tau|_\f^2\vol_{g_\f}$ and the identity \ref{idiv}) of Lemma \ref{su3identities}. 
\end{proof}

\medskip

We now examine an example of closed G$_2$-structure on the nilpotent Lie algebra $\frn_9$ in the light of the Propositions \ref{ClosedContactization} and \ref{invClosedContactization}. 

\begin{example}\label{exCT}
Consider the six-dimensional nilpotent Lie algebra $\mathfrak{h}$ with structure equations
\[
(0,0,  e^{12}, e^{13}, e^{23}, e^{15} + e^{24}). 
\]
The following 2-forms are symplectic forms on $\mathfrak{h}$:
\[
\omega_{\sst0} =  e^{16}  + e^{25} + e^{34}, \quad \tomega = - e^{12}- e^{14}  - e^{35} + e^{26}. 
\]
Let $\mathfrak{g} = \mathfrak{h} \oplus \mathbb{R} e_7$ be the contactization of $(\mathfrak{h},\omega_{\sst0})$ with contact form $\theta=e^7$ and Reeb vector $z = e_7$. 
Then, $\mathfrak{g}$ has structure equations
\[
(0,0,  e^{12}, e^{13}, e^{23}, e^{15} + e^{24}, e^{16}  + e^{25} + e^{34}),
\]
and so it coincides with the Lie algebra $\frn_9$ described in the proof of Corollary \ref{CorNil}. 
It is easy to check that $\tomega\wedge\omega_{\sst0} = - d\rho$, where 
\[
\rho = e^{124} - e^{125} - e^{136} - e^{234} - e^{345}  + e^{456}
\]
is a definite 3-form on $\mathfrak{h}$. 
The almost complex structure $J$ induced by $(\rho,e^{123456})$ on $\mathfrak{h}$ is given by 
\[
J e_1 = -e_4-e_5,~ J e_2 = e_6,~ J e_3 = e_2 - e_5,~ J e_4 = e_1-e_3-e_6,~ J e_5 = e_3 +e_6,~ Je_6 = -e_2.  
\]
The 2-form $\tomega$ is a taming form for $J,$ as for any non-zero vector $\xi = \xi^ke_k\in\mathfrak{g}$ we have 
\[
\tomega(\xi,J\xi) = \sum_{k=1}^6 (\xi^k)^2 + \xi^1\xi^6 + \xi^2\xi^5 - \xi^3\xi^6 - \xi^4\xi^5 > 0. 
\] 
Therefore, $\varphi = \tomega \wedge e^7 + \rho$ defines a closed G$_2$-structure on $\mathfrak{g}$. 
The metric induced by $\f$ has the following expression
\[
g_\f = 2^{\frac13}\left[ \sum_{i=1}^7 e^i\odot e^i + (e^1\odot e^6 + e^2\odot e^5 - e^3\odot e^6 - e^4 \odot e^5) + (e^2 -e^4+e^5)\odot e^7 \right],
\]
where $e^i\odot e^k = \frac12\left(e^i\otimes e^k + e^k \otimes e^i\right)$.  

\smallskip

On the other hand, we can start with the Lie algebra $\frg$ endowed with the closed G$_2$-structure $\f$ and consider the SU(3)-structure induced by it on the $g_\f$-orthogonal 
complement $\fru$ of the one-dimensional subspace generated by $e_7\in\frz(\frg)$. 
We have $u=|e_7|_\f = 2^{\frac16}$ and 
\[
\eta =  u^{-2}\,(e_7)^\flat = \frac12\left(e^2-e^4+e^5\right) + e^7. 
\]
The closed G$_2$-structure $\f$ can be written as $\f=u\,{\omega}\W\eta+\psip$, where the pair
\[
\begin{split}
\omega	& = u^{-1}\,\tomega  = 2^{-\frac16}\left( - e^{12}- e^{14}  - e^{35} + e^{26}\right),\\
\psip		&= -\frac12\,e^{125}-e^{136}+\frac12\,e^{145}-e^{234} -\frac12\,e^{246}+\frac12\,e^{256}-\frac12\,e^{345}+\frac12\,e^{235}+e^{456},
\end{split}
\]
defines an SU(3)-structure on the vector subspace $\fru\subset \frg$.  
Notice also that $\frh \coloneqq \ker \theta$ is a Lie algebra with respect to the bracket \eqref{bracketh}, and that it is endowed with an SU(3)-structure $(\omega,\psip)$ whose expression with respect to  
the basis $\{e^1,\ldots,e^6\}$ of $\frh^*$ is the same as the one appearing above (cf.~Remark \ref{remcontsu3}). 
The metric induced by $({\omega},\psip)$ on $\frh$ is 
\[
\begin{split}
g	&= {2}^{\frac13}\left(e^1\odot e^1 + e^3\odot e^3 + e^6\odot e^6\right) + \frac{3}{4}\, {2}^{\frac13} \left(e^2\odot e^2 + e^4\odot e^4 + e^5\odot e^5\right) \\
	&\quad +{2}^{\frac13} \left[e^1\odot e^6 - e^3\odot e^6 +\frac12\left(e^2\odot e^4 + e^2\odot e^5 - e^4\odot e^5\right)\right], 
\end{split}
\]
and we have $g_\f = g +u^{2}\,\eta\otimes\eta$. 
\end{example}

The results of Proposition \ref{ClosedContactization} are also useful to produce examples of seven-dimensional solvable non-nilpotent Lie algebras admitting closed G$_2$-structures, as the next example shows.

\begin{example}\label{ExS9}
On the six-dimensional unimodular solvable non-nilpotent Lie algebra $\mathfrak{g}_{6,70}^{0,0}$ with structure equations
\[
\left(-e^{26}+e^{35},e^{16}+e^{45},-e^{46},e^{36},0,0\right), 
\]
consider the closed 2-forms
\[
\omega_{\sst0} =2e^{34}, \quad \tomega = - e^{13}- e^{24}  - e^{56},
\]
and the definite 3-form
\[
\rho = e^{125} -e^{146}+e^{236}- e^{345}.
\]
Then, $d\rho=-{\tomega}\wedge\omegaz$, and the almost complex structure $J$ induced by the pair $(\rho,e^{123456})$ is given by
\[
J e_1 = -e_3,~ J e_2 = -e_4,~ J e_3 = e_1,~ J e_4 = e_2,~ J e_5 = -e_6,~ J e_6 = e_5.
\]
In particular, $\tomega$ is a taming form for $J$, as for any non-zero vector $\xi = \xi^ke_k\in\mathfrak{g}_{6,70}^{0,0}$ we have 
\[
\tomega(\xi,J\xi) = \sum_{k=1}^6 (\xi^k)^2> 0.  
\] 
The pair $( \tomega,\rho)$ defines an $\text{SU}(3)$-structure on $\mathfrak{g}_{6,70}^{0,0}$, 
since $ \tomega\wedge \rho=0$ and $3\rho\wedge J_{\rho}\rho = 2 \tomega^3$.

The central extension of $(\mathfrak{g}_{6,70}^{0,0},\omegaz)$ is given by 
\[
\frg = (-e^{26}+e^{35},e^{16}+e^{45},-e^{46},e^{36},0,0, 2e^{34}),
\] 
and it is isomorphic to the Lie algebra $\mathfrak{s}_9$ of Theorem \ref{classification} below. 
By Proposition \ref{ClosedContactization}, we know that the 3-form $\f = \tomega\W e^7+\rho$ defines a closed $\G_2$-structure on $\frg$. 
Notice that the 1-form $e^7$ does not define a contact structure on $\frg$, as the closed 2-form $\omegaz$ is degenerate.
\end{example}

More generally, one can consider the list of all six-dimensional unimodular solvable non-nilpotent Lie algebras admitting symplectic structures \cite{Mac} 
and determine which of them admit a structure $(\omegaz,\tomega,\rho)$ satisfying the hypothesis of Proposition \ref{ClosedContactization}. 
This allows one to obtain further examples of solvable non-nilpotent Lie algebras admitting closed G$_2$-structures.

%%%%%%%%%%%%%%%%%%%%%%%%%%%%%%%%%%%%%%%%%%%%%%%%%%%%%%%%%%%%%%%%%%%%%%%%%%%%%%%%%%%%%%%%%
%%%%%%%%%%%%%%%%%%%%%%%%%%%%%%%%%%%%%%%%%%%%%%%%%%%%%%%%%%%%%%%%%%%%%%%%%%%%%%%%%%%%%%%%%
%																CLASSIFICATION 
%%%%%%%%%%%%%%%%%%%%%%%%%%%%%%%%%%%%%%%%%%%%%%%%%%%%%%%%%%%%%%%%%%%%%%%%%%%%%%%%%%%%%%%%%
%%%%%%%%%%%%%%%%%%%%%%%%%%%%%%%%%%%%%%%%%%%%%%%%%%%%%%%%%%%%%%%%%%%%%%%%%%%%%%%%%%%%%%%%%

\section{A classification result}\label{ClassSect}

In this section, we classify all seven-dimensional unimodular Lie algebras with non-trivial center admitting closed G$_2$-structures, up to isomorphism. 
Every such Lie algebra must be solvable by the results of \cite{FiRa1}. If it is nilpotent, then it is isomorphic to one of the Lie algebras $\frn_1,\ldots,\frn_{12}$ by \cite{CoFe}. 
To complete the classification, we have to determine which unimodular solvable non-nilpotent Lie algebras with non-trivial center admit closed G$_2$-structures.  
This is done in the following. 

\begin{theorem}\label{classification}
Let $\mathfrak{g}$ be a seven-dimensional unimodular solvable non-nilpotent Lie algebra with non-trivial center. 
Then, $\mathfrak{g}$ admits closed $\G_2$-structures if and only if it is isomorphic to one of the following 
\[
\begin{split}
\frs_1	&=	(e^{23},-e^{36},e^{26},e^{26}-e^{56},e^{36}+e^{46},0,0), \\
\frs_2	&=	(e^{16}+e^{35},-e^{26}+e^{45},e^{36},-e^{46},0,0,0), \\
\frs_3	&=	(-e^{16}+e^{25},-e^{15}-e^{26},e^{36}-e^{45},e^{35}+e^{46},0,0,0), \\
\frs_4	&=	(0,-e^{13},-e^{12},0,-e^{46},-e^{45},0), \\
\frs_5	&=	(e^{15},-e^{25},-e^{35},e^{45},0,0,0), \\
\frs_6	&=	(\alpha e^{15}+e^{25},-e^{15}+\alpha e^{25},-\alpha e^{35}+e^{45},-e^{35}-\alpha e^{45},0,0,0),~ \alpha> 0, \\
\frs_7	&=	(e^{25},-e^{15},e^{45},-e^{35},0,0,0), \\
\frs_8	&=	(e^{16}+e^{35},-e^{26}+e^{45},e^{36},-e^{46},0,0,e^{34}), \\
\frs_9	&=	(-e^{26}+e^{35},e^{16}+e^{45},-e^{46},e^{36},0,0, e^{34}), \\
\frs_{10}	&=	\left(e^{23},-e^{36},e^{26},e^{26}-e^{56},e^{36}+e^{46},0,2\,e^{16}+e^{25}-e^{34}+\sqrt{3}\,e^{24} +\sqrt{3}\,e^{35}\right),\\
\frs_{11}	&=	\left(e^{23},-e^{36},e^{26},e^{26}-e^{56},e^{36}+e^{46},0,2\,e^{16}+e^{25}-e^{34}-\sqrt{3}\,e^{24}-\sqrt{3}\,e^{35}\right).
\end{split}
\]
In particular, $\frg$ is the contactization of a symplectic Lie algebra if and only if it is isomorphic either to $\frs_{10}$ or to $\frs_{11}$. 
\end{theorem}

\begin{proof}
Since the central extension of a nilpotent Lie algebra is nilpotent, by Proposition \ref{invClosedContactization} we can assume that $\frg$ is the central extension of a six-dimensional unimodular solvable non-nilpotent 
Lie algebra $(\frh,\omegaz)$ admitting symplectic structures. 
Recall that $\frg$ is determined by any representative in the cohomology class $[\omegaz]\in H^2(\mathfrak{h})$, up to isomorphism. 
Moreover, $\frh$ is isomorphic to one of the Lie algebras listed in Table \ref{table1}  (cf.~\cite{FerMan, Mac}).

If $\omegaz=0$, then $\frg$ is the direct sum of $\frh$ and the abelian Lie algebra $\R$.  
As a consequence of \cite[Thm.~C]{FinoSal}, $\frg$ admits closed $\G_2$-structures if and only if $\frh$ admits symplectic half-flat SU(3)-structures.   
Therefore, by \cite[Thm.~1.1]{FerMan}, $\frg$ must be isomorphic to one of the Lie algebras 
$\frs_1\cong\frg^0_{6,38}\oplus\R$, $\frs_2\cong \frg^{0,-1}_{6,54}\oplus\R$, $\frs_3\cong \frg^{0,-1,-1}_{6,118}\oplus\R$, $\frs_4\cong \fre(1,1)\oplus\fre(1,1)\oplus\R$, 
$\frs_5\cong A^{-1,-1,1}_{5,7}\oplus\R^2$, $\frs_6\cong A^{\alpha,-\alpha,1}_{5,17}\oplus\R^2$, $\frs_7\cong A^{0,0,1}_{5,17}\oplus\R^2$.  

We can then focus on the case when $\omegaz\neq0$ and $\frh$ is one of the Lie algebras appearing in Table \ref{table1}. 
To determine those having a central extension admitting closed $\G_2$-structure, we proceed as follows. 
First, we compute a basis of the second cohomology group $H^2(\frh)$ using the structure equations given in Table \ref{table1}. 
Then, we consider a non-zero representative $\omegaz$ of the generic element in $H^2(\frh)$, 
and we look for closed non-degenerate 2-forms $\tomega\in\Lambda^2\frh^*$ such that $\tomega\W\omegaz$ is exact (cf.~Proposition \ref{invClosedContactization}). 
A computation (with the aid of a computer algebra system, e.g.~Maple 2021, when needed) shows that 
there are no exact 4-forms of this type when $\frh$ is a decomposable Lie algebra not isomorphic to $A_{5,15}^{-1}\oplus\R$ or to $A_{5,18}^0\oplus\R$. 
Let us prove this claim, for instance, for the first decomposable Lie algebra appearing in Table \ref{table1}, namely $\frg_{6,13}^{-1,\frac{1}{2},0}$. 
A basis for its second cohomology group is given by
\[
\left( [e^{13}], [e^{24}], [e^{56}] \right), 
\] 
and we can consider the non-zero representative 
\[
\omegaz=f_1\,e^{13}+f_2\,e^{24}+f_3\, e^{56},
\] 
where $f_1,f_2,f_3\in \R$ and $f_1^2+f_2^2+f_3^2\neq0$.
The generic closed non-degenerate 2-form $\tomega$ on $\frg_{6,13}^{-1,\frac{1}{2},0}$ has the following expression
\[
\tomega=h_1\, e^{13}+h_2\left(e^{23}-\frac{1}{2}\, e^{16}\right)+h_3\, e^{24}+h_4\, e^{26}+h_5\, e^{36}+h_6\,e^{46}+h_7\, e^{56},
\] 
for some $h_i\in \R$ such that $h_1\,h_3\,h_7\neq 0$.  
Now, we compute  
\[
\begin{split}
\tomega \wedge \omegaz =&	-\left(f_1 h_3 + f_2 h_1\right) e^{1234} -f_1 h_4\, e^{1236} +f_1 h_6\, e^{1346} + \left(f_1 h_7 + f_3 h_1\right) e^{1356} \\
					&	-\frac12 f_2h_2\, e^{1246} -f_2h_5\,e^{2346} + \left(f_2h_7+f_3h_3\right) e^{2456}  + f_3h_2\,e^{2356}, 
\end{split} 
\]
and we see that this 4-form is exact only if the coefficients of $e^{1234}, e^{1356}, e^{2456}$ vanish, namely  
\[
\begin{cases}
	f_1\,h_3+f_2\,h_1=0, \\
	f_1\,h_7+f_3\,h_1=0, \\
	f_2\,h_7+f_3\,h_3=0.
\end{cases}
\]
This is a homogeneous linear system in the variables $f_i$'s whose unique solution under the constraint $h_1\,h_3\,h_7\neq 0$ is $f_1=f_2=f_3=0$. 
Thus, $\tomega\W\omegaz$ cannot be exact if $\omegaz\neq0$. 
A similar discussion leads us to ruling out all of the decomposable Lie algebras listed in Table \ref{table1} but $A_{5,15}^{-1}\oplus\R$ and $A_{5,18}^0\oplus\R$. 
In the remaining two cases, $\frh$ is the direct sum of a five-dimensional ideal $\frk$ and $\R$. 
A computation shows that there exist pairs $(\tomega,\omegaz)$ satisfying the required conditions only when $\omegaz\in\Lambda^2\frk^*$.  
In detail, if $\frh\cong A_{5,15}^{-1}\oplus\R$, then the possible 2-forms are given by
\[  
\tomega = h_1\,(e^{14}-e^{23})+h_2\,e^{15}+h_3\,e^{24}+h_4\,e^{25}+h_5\,e^{35}+h_6\,e^{45} + h_7\,e^{56},\quad \omegaz = a\,e^{24}, 
\]
where $a,h_i\in\R$ and $a\,h_1\,h_7\neq0$. If $\frh\cong A_{5,18}^0\oplus\R$, then the possible 2-forms are given by
\[
\tomega = k_1\,(e^{13}+e^{24})+k_2\, e^{15}+k_3\, e^{25}+k_4\, e^{34}+k_5\, e^{35}+k_6\, e^{45}+k_7\, e^{56},\quad \omegaz = b\,e^{34},
\]
where $b,k_i\in\R$ and $b\,k_1\,k_7\neq0$. 
Since in both cases $\frh=\frk\oplus\R$ and $\omegaz\in\Lambda^2\frk^*$, all possible central extensions of $(\frh,\omegaz)$ split as the Lie algebra direct sum of a six-dimensional ideal and $\R$. 
If such an extension admits closed $\G_2$-structures, then it must be isomorphic to one of $\frs_1,\ldots, \frs_7$.   

We are then left with the indecomposable Lie algebras appearing in Table \ref{table1}.
Also in this case, with analogous computations as before, one can check that there are no pairs $(\tomega,\omegaz)$ satisfying the required conditions 
when $\frh$ is an indecomposable Lie algebra not isomorphic to one of $\frg_{6,38}^0$, $\frg_{6,54}^{0,-1}$, $\frg_{6,70}^{0,0}$. 
In the remaining three cases, we claim that $\frh$ has a central extension admitting closed $\G_2$-structures. 

If $\frh\cong\frg_{6,38}^0$, then there exist closed non degenerate 2-forms $\tomega$ such that $\tomega\W\omegaz$ is exact if and only if either 
$\omegaz = a\left( 2\,e^{16}+e^{25}-e^{34}+\sqrt{3}\,e^{24} +\sqrt{3}\,e^{35}\right)$, for some $a\neq0$, or $\omegaz = b\left(2\,e^{16}+e^{25}-e^{34}-\sqrt{3}\,e^{24} -\sqrt{3}\,e^{35}\right)$, for some $b\neq0$. 
These forms are not cohomologous, so they give rise to non-isomorphic central extensions of $\frh$.  
In the first case, the central extension of $(\frh,\omegaz)$ is isomorphic to $\frs_{10}$, and it admits closed $\G_2$-structures. An example is given by
\[
\f = e^{123}-4\,e^{145}+2\, e^{167}-\sqrt{3}\,e^{247}+e^{256}+e^{257}-e^{346}-e^{347}-\sqrt{3}\,e^{357}.
\]
In the second case, the central extension of $(\frh,\omegaz)$ is isomorphic to $\frs_{11}$ and it admits closed $\G_2$-structures. An example is given by
\[
\f = e^{123} -4\, e^{145}+2\, e^{167}+\sqrt{3}\, e^{247}-e^{256}+e^{257}+e^{346}-e^{347}+\sqrt{3}\,e^{357}. 
\]
Both $\frs_{10}$ and $\frs_{11}$ are contactizations, since the 2-form $\omegaz$ is non-degenerate.   

If $\frh\cong \frg_{6,54}^{0,-1}$, then $\omegaz = a\,e^{34}$, for some $a\neq0$. 
The central extension of $(\frh,\omegaz)$ is isomorphic to $\frs_{8}$, and it admits closed $\G_2$-structures. An example is given by
\[
\f = e^{147}+e^{237}+e^{567}+e^{125}-e^{136}+\frac{1}{2}(e^{146}-e^{236})+\frac{5}{4}e^{246}+e^{345}.
\]
If $\frh\cong \frg_{6,70}^{0,0}$, then $\omegaz = a\,e^{34}$, for some $a\neq0$. The central extension of $(\frh,\omegaz)$ is 
isomorphic to $\frs_{9}$, and it admits closed $\G_2$-structures. An example is given by 
\[
\f = e^{137} +e^{247} +2\, e^{567} -e^{125} +e^{146} -e^{236} +e^{345},
\]
see also Example \ref{ExS9}. 

To conclude the proof, we first observe that the Lie algebras $\frs_8$ and $\frs_9$ cannot be the contactization of a symplectic Lie algebra. 
Indeed, in both cases $\omegaz$ is a closed degenerate 2-form on the unimodular Lie algebra $\frh$, thus every representative of $[\omegaz]\in H^2(\frh)$ is a degenerate 2-form.  
Finally, a direct computation shows that the remaining Lie algebras $\frs_1,\ldots,\frs_7$ do not admit any contact structure. 
\end{proof}

\begin{table}[H]
\begin{center}
\addtolength{\leftskip} {-3cm}
\addtolength{\rightskip}{-3cm}
\caption{Isomorphism classes of six-dimensional unimodular solvable non-nilpotent Lie algebras admitting symplectic structures \cite{FerMan, Mac}.}
\label{table1}
\smallskip
\scalebox{0.85}{
\renewcommand\arraystretch{1.25}
\begin{tabular}{|c|c|c|}
\hline
$\mathfrak{g}$ & {\normalfont Structure equations} 						& decomposable 	\\ \hline
$\mathfrak{g}_{6,3}^{0,-1}$ &$(e^{26},e^{36},0,e^{46},-e^{56},0)$ 			&  				\\ \hline 
$\mathfrak{g}_{6,10}^{0,0}$ &$(e^{26},e^{36},0,e^{56},-e^{46},0)$ 			& \\ \hline
$\mathfrak{g}_{6,13}^{-1,\frac{1}{2},0}$ &$(-\frac{1}{2}e^{16}+e^{23},-e^{26},\frac{1}{2}e^{36},e^{46},0,0)$		& \checkmark\\ \hline
$\mathfrak{g}_{6,13}^{\frac{1}{2},-1,0}$ &$(-\frac{1}{2}e^{16}+e^{23},\frac{1}{2}e^{26},-e^{36},e^{46},0,0)$		& \checkmark\\ \hline
$\mathfrak{g}_{6,15}^{-1}$ &$(e^{23},e^{26},-e^{36},e^{26}+e^{46},e^{36}-e^{56},0)$						& \\ \hline
$\mathfrak{g}_{6,18}^{-1,-1}$  &$(e^{23},-e^{26},e^{36},e^{36}+e^{46},-e^{56},0)$						& \\ \hline
$\mathfrak{g}_{6,21}^0$ & $(e^{23},0,e^{26},e^{46},-e^{56},0)$ 									& \\ \hline
$\mathfrak{g}_{6,36}^{0,0}$  &$(e^{23},0,e^{26},-e^{56},e^{46},0)$									& \\ \hline
$\mathfrak{g}_{6,38}^0$  &$(e^{23},-e^{36},e^{26},e^{26}-e^{56},e^{36}+e^{46},0)$						& \\ \hline
$\mathfrak{g}_{6,54}^{0,-1}$  &$(e^{16}+e^{35},-e^{26}+e^{45},e^{36},-e^{46},0,0)$						&\\ \hline
$\mathfrak{g}_{6,70}^{0,0}$  &$(-e^{26}+e^{35},e^{16}+e^{45},-e^{46},e^{36},0,0)$						& \\ \hline
$\mathfrak{g}_{6,78}$  &$(-e^{16}+e^{25},e^{45},e^{24}+e^{36}+e^{46},e^{46},-e^{56},0)$					& \\ \hline
$\mathfrak{g}_{6,118}^{0,-1,-1}$  &$(-e^{16}+e^{25},-e^{15}-e^{26},e^{36}-e^{45},e^{35}+e^{46},0,0)$		&\\ \hline
$\mathfrak{n}_{6,84}^{\pm 1}$  &$(-e^{45},-e^{15}-e^{36},-e^{14}+e^{26} \mp e^{56},e^{56},-e^{46},0)$		&\\ \hline
$\mathfrak{e}(2) \oplus \mathfrak{e}(2)$ &$(0,-e^{13},e^{12},0,-e^{46},e^{45})$  						& \checkmark\\ \hline 
$\mathfrak{e}(1,1) \oplus \mathfrak{e}(1,1)$ &$(0,-e^{13},-e^{12},0,-e^{46},-e^{45})$ 						& \checkmark\\ \hline 
$\mathfrak{e}(2) \oplus \mathbb{R}^3$ &$(0,-e^{13},e^{12},0,0,0)$  									& \checkmark\\ \hline 
$\mathfrak{e}(1,1) \oplus \mathbb{R}^3$ &$(0,-e^{13},-e^{12},0,0,0)$  								& \checkmark\\ \hline 
$\mathfrak{e}(2) \oplus \mathfrak{e}(1,1)$ &$(0,-e^{13},e^{12},0,-e^{46},-e^{45})$  						& \checkmark\\ \hline 
$\mathfrak{e}(2) \oplus \mathfrak{h}_3$ &$(0,-e^{13},e^{12},0,0,e^{45})$ 								& \checkmark\\ \hline 
$\mathfrak{e}(1,1) \oplus \mathfrak{h}_3$ &$(0,-e^{13},-e^{12},0,0,e^{45})$ 							& \checkmark\\ \hline 
$A_{5,7}^{-1,\beta,-\beta}\oplus \mathbb{R}$ & $(e^{15},-e^{25},\beta e^{35},-\beta e^{45},0,0), \quad 0<\beta<1$ 	& \checkmark\\ \hline
{$A_{5,7}^{-1,-1,1}\oplus \mathbb{R}$} & ${(e^{15},-e^{25},-e^{35},e^{45},0,0)}$ 					& \checkmark\\ \hline
$A_{5,8}^{-1} \oplus \mathbb{R}$  & $(e^{25},0,e^{35},-e^{45},0,0)$  										& \checkmark\\ \hline
$A_{5,13}^{-1,0,\gamma}\oplus\R$ & $(e^{15},-e^{25},\gamma e^{45},-\gamma e^{35},0,0),\quad \gamma>0$ 		& \checkmark\\ \hline
$A_{5,14}^0 \oplus \mathbb{R}$ & $(e^{25},0,e^{45},-e^{35},0,0)$ 										& \checkmark\\ \hline
$A_{5,15}^{-1} \oplus \mathbb{R}$  &$(e^{15}+e^{25},e^{25},-e^{35}+e^{45},-e^{45},0,0)$ 						& \checkmark\\ \hline
$A_{5,17}^{\alpha,-\alpha,1} \oplus \mathbb{R}$  &$(\alpha e^{15} + e^{25},-e^{15} + \alpha e^{25},-\alpha e^{35} + e^{45},- e^{35}-\alpha e^{45},0,0),\quad {\alpha> 0}$ 	& \checkmark\\ \hline
$A_{5,17}^{0,0,\gamma} \oplus \mathbb{R}$  &$(e^{25},-e^{15},\gamma e^{45},-\gamma e^{35},0,0),\quad 0<\gamma<1$ 										& \checkmark\\ \hline
${A_{5,17}^{0,0,1} \oplus \mathbb{R}}$  &${(e^{25},-e^{15}, e^{45},-e^{35},0,0)} $ 																& \checkmark\\ \hline
$A_{5,18}^{0} \oplus \mathbb{R}$  &$(e^{25}+e^{35},-e^{15}+e^{45},e^{45},-e^{35},0,0)$ 																	& \checkmark\\ \hline
$A_{5,19}^{-1,2} \oplus \mathbb{R}$  &$(-e^{15}+e^{23},e^{25},-2e^{35},2e^{45},0,0)$ 																	& \checkmark\\ \hline
\end{tabular} 
}
\end{center}
\end{table}
\renewcommand\arraystretch{1}

\begin{remark}
Notice that there are some misprints in \cite{Mac} that have been corrected in Table \ref{table1}, see also the appendix in \cite{FerMan}. 
\end{remark}

\begin{corollary}\label{TFClass}
A seven-dimensional Lie algebra with non-trivial center admitting torsion-free $\G_2$-structures is either abelian or isomorphic to $\frs_7$. 
\end{corollary}
\begin{proof}
Let $\frg$ be a seven-dimensional Lie algebra with non-trivial center endowed with a torsion-free G$_2$-structure $\f$. 
Then, the metric $g_\f$ induced by $\f$ is Ricci-flat, and thus flat by \cite{AlekKim}. 
Consequently, the results of \cite{Milnor} imply that either $\frg$ is abelian, or $\frg$ splits as a $g_\f$-orthogonal direct sum $\frg = \frb \oplus \fra$, where $\frb$ is an abelian subalgebra, 
$\fra$ is an abelian ideal, and the endomorphism $\ad_x$ is skew-adjoint for every $x\in \frb$. 
In the latter case, $\frg$ is a unimodular 2-step solvable Lie algebra and the eigenvalues of $\ad_x$ are purely imaginary for every $x\in\frg$ (cf.~\cite[Sect.~2.8]{GoOn}). 
Among the Lie algebras obtained in Theorem \ref{classification}, the 2-step solvable ones are $\frs_2$, $\frs_3$, $\frs_4$, $\frs_5$, $\frs_6$, $\frs_7$. 
The first four Lie algebras in this list do not admit any flat metric, as  the following endomorphisms have real spectrum: $\ad_{e_6}\in\Der(\frs_2)$, $\ad_{e_6}\in\Der(\frs_3)$, 
$\ad_{e_1}\in\Der(\frs_4)$, $\ad_{e_5}\in\Der(\frs_5)$. Also the Lie algebra $\frs_6$ can be ruled out, since $\ad_{e_5}$ has complex eigenvalues that are not purely imaginary. 
Finally, the Lie algebra $\frs_7$ admits torsion-free G$_2$-structures. An example is given by the G$_2$-structure
\[
\f = e^{137} + e^{247} + e^{567} + e^{125} - e^{146} +e^{236} -e^{345},
\]
which induces the metric $g_\f = \sum_{i=1}^7 e^i \odot e^i$. 
\end{proof}

\begin{remark}\label{latticesconstr}
The simply connected solvable Lie groups whose Lie algebra is one of $\frs_1,\frs_2,\frs_3,\frs_4,\frs_5,\frs_7$ admit lattices, 
and this is the case also for the family of Lie algebras  $\frs_6\cong A^{\alpha,-\alpha,1}_{5,17}\oplus\R^2$, for certain values of the parameter $\alpha>0$ (see e.g.~\cite{FerMan} and the references therein). 
We now show that the simply connected Lie groups with Lie algebra $\frs_8$ or $\frs_9$ admit lattices, too. 
Indeed, since they are both almost nilpotent, it is possible to construct a lattice using the following criterion by \cite{Bock}.
Let $\G=\R \ltimes_{\mu} \mathrm{H}$ be an almost nilpotent Lie group with nilradical $\mathrm{H}$, and let $\frg = \R \ltimes_D \frh$ be its Lie algebra, where $\frh\coloneqq \mathrm{Lie}(\mathrm{H})$ and 
$D \in \Der(\frh)$ is such that $d\mu(t)|_{1_\G}=\exp(tD)$. If there exists $0\neq t_0\in \mathbb{R}$ and a rational
basis $(x_1, \ldots, x_n)$ of $\frh$ such that the coordinate matrix of $\exp(t_0 D)$ in such a basis is
integer, then $\Gamma\coloneqq t_0\mathbb{Z} \ltimes_{\mu} \text{exp} (\mathbb{Z}\langle x_1,\ldots,x_n \rangle)$ is a lattice in $\G$.

Let us consider the Lie algebra $\frs_9$ with the basis $(e_1,\ldots,e_7)$ as in Theorem \ref{classification}. 
We can write $\frs_9=\mathbb{R}\ltimes_D \frh$, where $\frh= \span_\R( e_1,\ldots, e_5, e_7)$ is a nilpotent Lie algebra with structure equations 
\begin{equation}\label{h}
(e^{35},e^{45},0,0,0,e^{34}),
\end{equation}
and 
\[
D=\mathrm{ad}(e_6)\rvert_{\frh}=
\begin{pmatrix}
0 & -1 & 0 & 0 & 0 & 0 \\ 
1 & 0 & 0 & 0 & 0 & 0 \\ 
0 & 0 & 0 & -1 & 0 & 0 \\ 
0 & 0 & 1 & 0 & 0 & 0 \\ 
0 & 0 & 0 & 0 & 0 & 0 \\ 
0 & 0 & 0 & 0 & 0 & 0
\end{pmatrix}.
\]
For $t_0=2\pi$, this basis satisfies the criterion of \cite{Bock} guaranteeing the existence of a lattice in the simply connected Lie group corresponding to $\frs_9$. 

Let us now focus on the Lie algebra $\frs_8$ with the basis $(e_1,\ldots,e_7)$ as in Theorem \ref{classification}. 
We note that $\frs_8={\R}\ltimes_D \frh$, where the structure equations of the nilpotent Lie algebra $\frh=\span_\R(e_1,\ldots,e_5,e_7)$ are those given in \eqref{h}, and
$D={\ad}(e_6)\rvert_{\frh}=\mathrm{diag}(1,-1,1,-1,0,0)$.
Let $t_0=\ln\left( \frac{3+\sqrt{5}}{2} \right)$. We note that $\exp(t_0 D)=E^{-1}B E$, where
\[
B=
\begin{pmatrix}
3 & 0 & -1 & 0 & 0 & 0 \\ 
0 & 3 & 0 & -1 & 0 & 0 \\ 
1 & 0 & 0 & 0 & 0 & 0 \\ 
0 & 1 & 0 & 0 & 0 & 0 \\ 
0 & 0 & 0 & 0 & 1 & 0 \\ 
0 & 0 & 0 & 0 & 0 & 1
\end{pmatrix}, \quad 
E=\begin{pmatrix}
\frac{2}{3-\sqrt{5}} & \frac{2}{3+\sqrt{5}} & 0 & 0 & 0 & 0 \\ 
0 & 0 & \frac{2}{3-\sqrt{5}} & \frac{2}{3+\sqrt{5}} & 0 & 0 \\ 
1 & 1 & 0 & 0 & 0 & 0 \\ 
0 & 0 & 1 & 1 & 0 & 0 \\ 
0 & 0 & 0 & 0 & 1 & 0 \\ 
0 & 0 & 0 & 0 & 0 & \sqrt{5}
\end{pmatrix}. 
\] 
Thus, the integer matrix $B$ is the matrix associated with $\exp(t_0D)$ with respect to a suitable basis $(\varepsilon_1,\ldots,\varepsilon_5,\varepsilon_7)$ of $\frh$.   
Moreover, $\frh$ has rational structure equations $(\varepsilon^{25},0,\varepsilon^{45},0,0,\varepsilon^{24})$ in such a basis. 
The existence of a lattice in the simply connected solvable Lie group with Lie algebra $\frs_8$ then follows. 
\end{remark}

%%%%%%%%%%%%%%%%%%%%%%%%%%%%%%%%%%%%%%%%%%%%%%%%%%%%%%%%%%%%%%%%%%%%%%%%%%%%%%%%%%%%%%%%%
%%%%%%%%%%%%%%%%%%%%%%%%%%%%%%%%%%%%%%%%%%%%%%%%%%%%%%%%%%%%%%%%%%%%%%%%%%%%%%%%%%%%%%%%%
%																SOLITONS 
%%%%%%%%%%%%%%%%%%%%%%%%%%%%%%%%%%%%%%%%%%%%%%%%%%%%%%%%%%%%%%%%%%%%%%%%%%%%%%%%%%%%%%%%%
%%%%%%%%%%%%%%%%%%%%%%%%%%%%%%%%%%%%%%%%%%%%%%%%%%%%%%%%%%%%%%%%%%%%%%%%%%%%%%%%%%%%%%%%%

\section{Semi-algebraic solitons on the central extension of a Lie algebra}\label{SolitonsSect}

Let $\G$ be a seven-dimensional simply connected Lie group with Lie algebra $\frg$. Consider a derivation $D$ of $\frg$, 
and denote by $X_D$ the vector field on $\G$ induced by the one-parameter group of automorphisms $F_t\in\mathrm{Aut}(\G)$ with derivative $dF_t|_{1_\G} = \exp(tD)\in \text{Aut}(\mathfrak{g})$. 
According to \cite{Lauret}, a left-invariant closed $\G_2$-structure $\f$ on $\G$ is said to be a \emph{semi-algebraic soliton} if it satisfies the Laplacian soliton equation \eqref{LapSolIntro} with respect to 
some vector field $X_D$ corresponding to a derivation $D\in\Der(\frg)$. In this case, $\mathcal{L}_{X_{D}}\f = D^*\f$, so that the equation \eqref{LapSolIntro} can be rewritten as follows  
\begin{equation}\label{semi-algebraic}
\Delta_{\varphi}\varphi=\lambda \varphi + D^*\varphi,
\end{equation}
where  
\[
A^*\beta(x_1,\ldots,x_k)=\beta(A x_1,\ldots,x_k)+\cdots+\beta(x_1,\ldots,A x_k),
\]  
for any $A\in\mathrm{End}(\frg)$, $x_1,\ldots,x_k\in \mathfrak{g}$, and $\beta\in \Lambda^k \mathfrak{g}^*$. Notice that $\Delta_\f\f = dd^*\f = d\tau,$ $\tau$ being the intrinsic torsion form of $\f$. 
When the $g_{\varphi}$-adjoint $D^t$ of $D$ is also a derivation of $\mathfrak{g}$, the $\text{G}_2$-structure $\varphi$ is called an \emph{algebraic soliton}.

\smallskip

We now focus on the case when $\mathfrak{g}$ is a unimodular Lie algebra with non-trivial center. 
By the results of Sect.~\ref{ClosedG2CExt}, we can assume that $\frg = \frh\oplus \R z$ is the central extension of a six-dimensional unimodular Lie algebra $(\mathfrak{h},\omegaz)$. 
Moreover, every closed $\G_2$-structure $\f$ on $\frg$ can be written both as $\f = \tomega\W\theta+\rho$, 
with $d\rho = -\tomega\W\omegaz$ and $d\tomega=0$, and as $\f = u\,\omega\W\eta+\psip,$ where  $(\omega,\psip)$  is an SU(3)-structure on $\frh$, 
$u\coloneqq |z|_\f$ and $\eta\coloneqq u^{-2} z^\flat = \eta_\frh+\theta$, for some $\eta_\frh\in\frh^*$. 

If $\f$ is a semi-algebraic soliton, the condition \eqref{semi-algebraic} is equivalent to a set of equations involving either the forms $(\tomega,\rho)$ or the 
SU(3)-structure $(\omega,\psip)$ on $\frh$. 
In the following, we shall see that it is possible to obtain information on the semi-algebraic soliton $\f$ under suitable assumptions. 
We are interested in the case when $z$ is an eigenvector of $D$, as this happens whenever $\frg$ is the contactization of a symplectic Lie algebra. 
Indeed, in that case the center of $\frg$ is $\frz(\frg)= \R z$ and it is preserved by any derivation of $\frg$. 

\smallskip

Henceforth, we assume that $\f$ is a semi-algebraic soliton on the unimodular Lie algebra $\frg=\frh\oplus\R z$, 
and that it satisfies the equation $\Delta_\f \f= \lambda\f+D^*\f$ with respect to a derivation $D\in\Der(\frg)$ such that 
$Dz = c z$, for some $c\in\R$. 
Then, we have $D^*\theta = \alpha + c\,\theta,$ with $\alpha\in\frh^*$. 
We let $\widetilde{D}\coloneqq \pi_\frh \circ D|_{\frh} \in\End(\frh)$, where $\pi_\frh:\frg\rightarrow \frh$ denotes the projection onto $\frh$. 

\smallskip

Using the expression of $\tau$ obtained in Lemma \ref{tauLemma}, we see that $\f = \tomega\W\theta+\rho= u\,\omega\W\eta+\psip$ solves the equation $d\tau = \Delta_\f\f=\lambda\f+D^*\f$ 
if and only if the following equations hold on $\frh$
\begin{equation}\label{solitonh}
\begin{cases}
2\,d(J\Wu) = - \widetilde{D}^*\omega - \left(c+\lambda\right)\omega,\\   
d\Wm -d*_g(J\Wu\W\psip) -2u \left(d(J\Wu)\W\eta_\frh - J\Wu\W d\eta\right) = u\,\omega\W\alpha + \widetilde{D}^*\rho + \lambda\rho, 
\end{cases}
\end{equation}
where $\Wu,\Wm$ are the intrinsic torsion forms of the SU(3)-structure $(\omega,\psip)$ and $\rho = u\,\omega\W\eta_\frh+\psip$. 
Recall that the 2-form $d\eta$ depends on the intrinsic torsion forms $\Wp$ and $\Wu$ of $(\omega,\psip)$ as shown in Lemma \ref{lemdeta}. 

\smallskip

The equations \eqref{solitonh} allow us to relate the constant $\lambda$ to the eigenvalue $c$ and the norm of the intrinsic torsion form of the semi-algebraic soliton. 
Before stating the result, we show a preliminary lemma. 

\begin{lemma}\label{EndPsi}
Let $(\omega,\psip)$ be an $\SU(3)$-structure on a six-dimensional vector space $V$ and let $A\in \End(V)$.
Then, 
\[
A^*\psip \wedge \psim = A^*\omega\wedge \omega^2 = \frac{1}{3}\, \tr(A)\,\omega^3.
\]
\end{lemma}

\begin{proof}
We can always consider a basis $(e_1,\ldots,e_6)$ of $V$ which is adapted to the $\SU(3)$-structure $(\omega,\psip)$. 
Then, with respect to the dual basis $(e^1,\ldots,e^6)$ of $V^*$, we have
\[
\omega=e^{12}+e^{34}+e^{56}, \quad \psip=e^{135}-e^{146}-e^{236}-e^{245}, \quad \psim=e^{136}+e^{145}+e^{235}-e^{246}.
\]
Now, a direct computation shows that
\[
A^*\psip \wedge \psim = A^*\omega\W\omega^2 = 2\, \tr(A)\, e^{123456} = \frac13\, \tr(A)\,\omega^3. 
\]
\end{proof}

\begin{proposition}\label{signLambda}
The constant $\lambda$ is given by 
\begin{equation}\label{LambdaCentExt}
\lambda = -3\,c-\frac12\left(|\Wm|_g^2 +6\,|\Wu|_g^2\right) = -3\,c -\frac12\,|\tau|_\f^2.
\end{equation}
\end{proposition}

\begin{proof}
Wedging the first equation of \eqref{solitonh} by the closed 4-form $\omega^2$, we obtain 
\[
\widetilde{D}^*\omega\wedge \omega^2+(c+\lambda)\omega^3  = -2\, d(Jw_1)\wedge \omega^2 =- 2\,d(Jw_1\wedge \omega^2) = 0,
\]
as every 5-form on the unimodular Lie algebra $\frh$ is closed. 
Then, by Lemma  \ref{EndPsi}, we get
\[
\tr(\widetilde{D}) = -3(c+\lambda).
\]

Let us now consider the second equation in \eqref{solitonh}. Wedging both sides by $\psim$ and using the compatibility condition $\omega\W\psim=0$, we obtain
\begin{equation}\label{auxcomp}
\left(d\Wm -d*_g(J\Wu\W\psip) -2u \left(d(J\Wu)\W\eta_\frh - J\Wu\W d\eta\right)\right)\W\psim = \left( \widetilde{D}^*\rho + \lambda\rho\right)\W\psim.
\end{equation}
Since $\rho = u\,\omega\W\eta_\frh+\psip$ and $\omega\W\psim=0$, the RHS of \eqref{auxcomp} can be rewritten as follows
\[
\begin{split}
\left( \widetilde{D}^*\rho + \lambda\rho\right)\W\psim 	&= u\,\widetilde{D}^*\omega\W\eta_\frh \W\psim +\widetilde{D}^*\psip \W\psim + \lambda\psip\W\psim\\
										&= u\left(-2\,d(J\Wu)- \left(c+\lambda\right)\omega \right)\W\eta_\frh \W\psim +\frac13\,\tr(\widetilde{D})\, \omega^3 + \frac23\,\lambda\,\omega^3\\
										&= -2u\,d(J\Wu)\W\eta_\frh \W\psim +\frac13\,\tr(\widetilde{D})\, \omega^3 + \frac23\,\lambda\,\omega^3,
\end{split}
\]
where the second equality follows from the first equation of \eqref{solitonh}, Lemma \ref{EndPsi} and the normalization condition $\psip\W\psim=\frac23\omega^3$. 
The summands appearing in the LHS of \eqref{auxcomp} can be rewritten as follows. 
Since $\Wm\in\Lambda^2_8\frh^*$, we have
\[
d\Wm\W\psim = d(\Wm\W\psim) - \Wm\W d\psim = -\Wm\W\left(-*_g\Wm +J\Wu\W\psip \right) = |\Wm|_g^2\,\vol_g.
\]
Since every 5-form on $\frh$ is closed, we get
\[
\begin{split}
-d*_g(J\Wu\W\psip)\W\psim 	&= *_g(J\Wu\W\psip)\W d\psim = *_g\left(J\Wu\W\psip\right) \W  \left(-*_g\Wm+J\Wu\W\psip\right) \\    
						&=   J\Wu\W 2*_g(J\Wu) = 2\,|J\Wu|_g^2\, \vol_g = 2\,|\Wu|_g^2\, \vol_g,
\end{split}
\]
where we used the identity \ref{idiv}) of Lemma \ref{su3identities}. Finally, by Lemma \ref{lemdeta} and the identity \ref{idiii}) of Lemma \ref{su3identities}, we have
\[
\begin{split}
2u\,J\Wu\W d\eta \W\psim	&= -2\,J\Wu \W *_g(\Wu\W\psip)\W\psim  = 4\,J\Wu\W*_g(J\Wu) = 4\,|\Wu|_g^2\,\vol_g.  
\end{split}
\]
Hence, the equation \eqref{auxcomp} becomes
\[
\left(|\Wm|_g^2 +6\,|\Wu|_g^2\right) \vol_g = \frac13\,\tr(\widetilde{D})\, \omega^3 + \frac23\,\lambda\,\omega^3.
\]
Recalling that $\vol_g =\frac16\,\omega^3$, we have
\[
\frac12\,\left(|\Wm|_g^2 +6\,|\Wu|_g^2\right) = \tr(\widetilde{D}) + 2\,\lambda.
\]
Now, the thesis follows combining this identity with $\tr(\widetilde{D}) = -3(c+\lambda)$ and recalling that  $|\Wm|_g^2 + 6\,|\Wu|_g^2 = |\tau|_\f^2$  (cf.~Lemma \ref{tauLemma}). 
\end{proof}

As a consequence of Proposition \ref{signLambda}, we have the following. 

\begin{corollary}\label{lambdacont}
Let $(\frg,\theta)$ be the contactization of a symplectic unimodular Lie algebra $(\frh,\omegaz)$, and let $\f$ be a semi-algebraic soliton on $\frg$ such that $\Delta_\f\f=\lambda\f+ D^*\f$, for some $D\in\Der(\frg)$. 
Then, 
\[
\lambda = |\Wm|_g^2 + 6\,|\Wu|_g^2 = |\tau|_\f^2, 
\]
and $\f$ is expanding.  
\end{corollary}

\begin{proof}
Since $\frg$ is the contactization of a symplectic Lie algebra $(\frh,\omegaz)$, we have $\frg = \frh \oplus\R z$ and $\frz(\frg)=\R z$. 
In particular, $Dz = cz,$ for some $c\in\R$. Therefore, by Proposition \ref{signLambda}, the constant $\lambda$ is given by
\[
\lambda = -3\,c-\frac12\left(|\Wm|_g^2 +6\,|\Wu|_g^2\right) = -3\,c -\frac12\,|\tau|_\f^2.
\]
Recall that $\omegaz=d\theta$ on $\frg$. Since $D\in\Der(\frg)$, we see that
\[
D^*\omegaz = D^*(d\theta) = d(D^*\theta) = d(\alpha + c\,\theta) = d\alpha + c\,\omegaz. 
\]
On the other hand, $\omegaz$ is a non-degenerate 2-form on the unimodular Lie algebra $\frh$. Consequently,
\[
 \frac13\,\tr(\widetilde{D})\,\omegaz^3 = \widetilde{D}^*\omegaz\W\omegaz^2 =  D^*\omegaz\W\omegaz^2 = (d\alpha + c\,\omegaz)\W\omegaz^2 = c\,\omegaz^3, 
\]
as every 5-form on $\frh$ is closed. 
Now, from the proof of Proposition \ref{signLambda}, we know that $3c = \tr(\widetilde{D}) = -3c-3\lambda$, whence $-2c = \lambda$. Therefore, we have 
$\lambda = |\Wm|_g^2 + 6\,|\Wu|_g^2 = |\tau|_\f^2$.  

To conclude the proof, we observe that $\lambda = 0$ if and only if $\f$ is torsion-free. 
By Corollary \ref{TFClass}, torsion-free $\G_2$-structures do not occur on the contactization of any symplectic unimodular Lie algebra. Thus, $\lambda>0$ and $\f$ is expanding. 
\end{proof}

The previous result applies, for instance, to the nilpotent Lie algebra $\frn_{12}$ endowed with the closed G$_2$-structure considered in \cite[Thm,~3.6]{FFM}.
\begin{example}\label{n12}
Consider the nilpotent Lie algebra $\frn_{12}$, and let $(e^1,\ldots,e^7)$ be the basis of $\frn_{12}^*$ for which the structure equations are the following
\[
\left(0,0,0,\frac{\sqrt{3}}{6}e^{12}, \frac{\sqrt{3}}{12}e^{13}-\frac{1}{4}e^{23}, -\frac{\sqrt{3}}{12}e^{23}-\frac{1}{4}e^{13},
\frac{\sqrt{3}}{12}e^{16}-\frac{1}{4}e^{15}+\frac{\sqrt{3}}{12}e^{25}+\frac{1}{4}e^{26} -\frac{\sqrt{3}}{6}e^{34}\right). 
\]
Recall that $\frn_{12}$ is the contactization of a six-dimensional symplectic nilpotent Lie algebra (cf.~Corollary \ref{CorNil}). 
The 3-form 
\[
\f = e^{167} +e^{257}+e^{347}+e^{135} -e^{124}-e^{236}-e^{456}  
\]
defines a closed G$_2$-structure on $\frn_{12}$ inducing the metric $g_\f = \sum_{i=1}^7 e^i \odot e^i$. The corresponding intrinsic torsion form is $\tau = \frac12 \left( e^{56} - e^{37} \right)$. 
A computation shows that $\f$ is an expanding algebraic soliton solving the equation $\Delta_\f\f = \lambda \f + D^*\f$ with $\lambda=\frac12 = |\tau|_\f^2$ and 
\[
D = -\frac18\,\diag\left(1,1,0,2,1,1,2\right) \in \Der(\frn_{12}). 
\]
\end{example}    

In addition to $\frn_{12}$, also the non-abelian nilpotent Lie algebras $\frn_2,\ldots,\frn_7$ admit (semi)-algebraic solitons (see \cite{Nicolini}). 
However, it is currently not known whether semi-algebraic solitons occur on the nilpotent Lie algebras $\frn_8$, $\frn_9$, $\frn_{10}$ and $\frn_{11}$. 

Using Corollary \ref{lambdacont} and Proposition \ref{invClosedContactization}, we can show that semi-algebraic solitons do not exist on $\frn_9$.

\begin{proposition}\label{n9} 
The nilpotent Lie algebra $\frn_9$ does not admit any semi-algebraic soliton.
\end{proposition}

\begin{proof}
As we observed in Corollary \ref{CorNil}, the Lie algebra $\frn_9$ is the contactization of the nilpotent Lie algebra 
$\frh =(0,0,e^{12},e^{13},e^{23},e^{15} + e^{24})$ endowed with the symplectic form $\omegaz =e^{16} + e^{34} + e^{25}$. 
In particular, $\frz(\frn_9) = \R e_7$. 

By Proposition \ref{invClosedContactization}, every closed G$_2$-structure on $\frn_9$ can be written as $\f = \tomega\W e^7+\rho,$ 
where  $\rho$ is a definite 3-form on $\frh$,  $\tomega$ is a taming symplectic form for $J_\rho$ and $d  \rho = - \tomega \wedge \omegaz$ (see Example \ref{exCT} for an explicit case). 
This G$_2$-structure is a semi-algebraic soliton solving $\Delta_\f\f=\lambda\f+D^*\f$, for some $\lambda\in\R$ and some $D\in\Der(\frn_9)$, 
if and only if the equations \eqref{solitonh} are satisfied. By Corollary \ref{lambdacont}, we must have $De_7=c e_7$ and $\lambda = -2c = |\tau|_\f^2>0$, 
Moreover, the first equation in \eqref{solitonh} constrains the 2-form on $\frh$ 
\[
\beta\coloneqq \widetilde{D}^*\tomega+\left(c+\lambda\right)\tomega = \widetilde{D}^*\tomega -c\,\tomega 
\]
to be exact. 
We shall show that if this last condition holds for any derivation $D$ and any symplectic form $\tomega$ on $\frn_9$, then $\tomega\W\omegaz$ cannot be exact. 

The matrix associated with the generic derivation $D\in\Der(\frn_9)$ with respect to the basis $(e_1,\ldots, e_7)$ is given by
\[
D=\begin{pmatrix}
h_1 & 0 & 0 & 0 & 0 & 0 & 0 \\ 
0 & 2h_1 & 0 & 0 & 0 & 0 & 0 \\ 
h_2 & h_3 & 3h_1 & 0 & 0 & 0 & 0 \\ 
h_4 & h_5 & h_3 & 4h_1 & 0 & 0 & 0 \\ 
h_6 & h_7 & -h_2 & 0 & 5h_1 & 0 & 0 \\ 
h_8 & h_9 & h_7-h_4 & -h_2 & h_3 & 6h_1 & 0 \\ 
h_{10} & h_{11} & h_9-h_6 & h_7-2h_4 & -h_5-h_2 & h_3 & 7h_1
\end{pmatrix},
\]
where $h_j\in \mathbb{R}$.  In particular, $c=7h_1$ and we can assume that $h_1<0$.

The generic closed 2-form $\tomega$ on $\frh$ is 
\[
\tomega 	= f_1 e^{12} + f_2 e^{13} + f_3 e^{14} + f_4 \left(e^{15} + e^{24} \right) + f_5 \left( e^{16} + e^{34} \right) + f_6 e ^{23} + f_7 e^{25} + f_{8} \left(e^{26} -  e^{35}\right), 
\]
and it is non-degenerate if and only if $f_5^2f_7 - f_3f_8^2\neq 0$. 
Now, we have
\[
\begin{split}
\beta	&=\widetilde{D}^*{\tomega}-7h_1\,{\tomega}\\ 
	&=\left(h_3f_2-4\,h_1f_1+h_5f_3-h_4f_4+h_7f_4+h_9f_5-h_2f_6-h_6f_7-h_8f_8  \right) e^{12} \\
&\quad + \left(h_3f_3-3\,h_1f_2-h_2f_4-2\,h_4f_5+f_5h_7+h_6f_8 \right)e^{13} +\left(h_3f_5 -h_1f_4 - h_2f_8 \right) \left(e^{15}+e^{24} \right) \\
&\quad +\left(h_3f_4-h_5f_5-2\,h_1f_6-h_2f_7-f_8h_4+2\,h_7f_8 \right) e^{23}\\
&\quad -2h_1f_3\, e^{14} + h_{1} f_{8}\, e^{26} - h_1 f_{8}\,e^{35}. 
\end{split}
\]
Since the space of exact $2$-forms on $\mathfrak{h}$ is generated by $e^{12}, e^{13}, e^{15}+e^{24},  e^{23}$, and since $h_1<0$, we see that $\beta$ is exact if and only if $f_3=0=f_8$. 
Consequently, $\tomega$ is non-degenerate if and only if $f_5 f_7\neq0$. This last constraint implies that $\tomega\W\omegaz$ cannot be exact, 
since the space of exact 4-forms is spanned by  
\[
e^{1234}, e^{1235},  e^{1245}, e^{1236}, e^{1246}+e^{1345}, e^{1256}-e^{2345},  e^{1356}-e^{2346}. 
\]
\end{proof}

\begin{remark}\label{s8s9}
Using a similar argument involving equations \eqref{solitonh} as in the proof of the last proposition, 
one can show that semi-algebraic solitons do not occur on the solvable non-nilpotent Lie algebras $\frs_8$ and $\frs_9$. 
\end{remark}

By \cite{CoFe} and Theorem \ref{classification}, we know that a seven-dimensional unimodular Lie algebra with one-dimensional center  admitting closed G$_2$-structures is isomorphic to one of  
$\frn_8$, $\frn_9$, $\frn_{10}$, $\frn_{11}$, $\frn_{12}$, $\frs_2$, $\frs_3$, $\frs_4$, $\frs_8$, $\frs_9$, $\frs_{10}$, $\frs_{11}$. 
By Proposition \ref{n9} and Remark \ref{s8s9}, the only ones which may admit semi-algebraic solitons are
$\frn_8$, $\frn_{10}$, $\frn_{11}$, $\frn_{12}$, $\frs_2$, $\frs_3$, $\frs_4$, $\frs_{10}$, $\frs_{11}$. 

Examples of expanding semi-algebraic solitons are known on $\frn_{12}$ (see Example \ref{n12}), and on the Lie algebras $\frs_2$ and $\frs_4$ (see \cite[Prop.~6.5]{FiRa}). 
In the remaining cases, it is still not known whether semi-algebraic solitons exist. However, if there is any, it must be expanding. 
This follows from Proposition \ref{lambdacont} when the Lie algebra is one of $\frn_{10}$, $\frn_{11}$, $\frn_{12}$, $\frs_{10}$, $\frs_{11}$,  
while it follows from a direct computation involving equations \eqref{solitonh} when the Lie algebra is one of $\frn_8$, $\frs_2$, $\frs_3$ and $\frs_4$. 

\medskip

When the center of the Lie algebra is at least two-dimensional, we have the following classification result. 

\begin{theorem}\label{SASClassZ2}
Let $\frg$ be a seven-dimensional unimodular Lie algebra with $\dim\frz(\frg)\geq 2$ admitting closed $\G_2$-structures.
Then, $\frg$ admits semi-algebraic solitons if and only if it is isomorphic to one of $\frn_1$, $\frn_2$, $\frn_3$, $\frn_4$, $\frn_5$, $\frn_6$, $\frn_7$, $\frs_5$, $\frs_6$, $\frs_7$.
\end{theorem}

\begin{proof}  
If $\frg$ is nilpotent, then it must be isomorphic to one of $\frn_1$, $\frn_2$, $\frn_3$, $\frn_4$, $\frn_5$, $\frn_6$, $\frn_7$, by the classification of \cite{CoFe}. 
Every $\G_2$-structure $\f$ on the abelian Lie algebra $\frn_1$ is torsion-free, and thus it solves the equation $\Delta_\f\f = \lambda \f + D^*\f$ with $\lambda=0$ and $D=0\in\Der(\frn_1)$. 
In the remaining cases, the existence of semi-algebraic solitons is known from \cite{Nicolini}. 

We can then focus on the case when $\frg$ is solvable non-nilpotent. By Theorem \ref{classification}, $\frg$ must be isomorphic to one of $\frs_1$, $\frs_5$, $\frs_6$, $\frs_7$.
Examples of semi-algebraic on $\frs_5$ and $\frs_6$ were given in \cite[Prop.~6.5]{FiRa}. 
By Corollary \ref{TFClass}, the Lie algebra $\frs_7$ admits torsion-free G$_2$-structures, which are semi-algebraic solitons with $\lambda=0$ and $D=0\in\Der(\frs_7)$.  

To conclude the proof, we must show that the Lie algebra $\frs_1$ does not admit any semi-algebraic soliton.     
Let us assume by contradiction that $\f$ is a semi-algebraic soliton on $\frs_1$. Then, as $\frs_1 \cong \frg_{6,38}^0\oplus \mathbb{R}$, 
we can write $\f={\tomega}\wedge e^7+\rho$, where $e^7$ spans $\R^*$, and ${\tomega}$ and $\rho$ are closed forms on $\frg_{6,38}^0$. In particular, we have 
\[
\begin{split}
\tomega &= f_1 \left (2\,e^{16}+e^{25}-e^{34} \right)  +f_2\,e^{23} + f_3 \left(e^{24} +e^{35}\right) +f_4\,e^{26} +f_5\,e^{36} + f_6\,e^{46} + f_7\,e^{56},\\ 
\rho 	&= p_1\,e^{123} + p_2\left(e^{124} +e^{135} \right) +p_3\,e^{126} + p_4\,e^{136} + p_5\left(e^{146}- e^{235} \right) + p_6\left( e^{156} + e^{234}\right) \\
	&\quad +p_7\,e^{236} + p_8\,e^{246} + p_9\,e^{256} +p_{10}\,e^{346} + p_{11}\,e^{356}+p_{12}\,e^{456},
\end{split}
\]
where the $f_i$'s and $p_j$'s are real parameters. 
The symmetric bilinear form $b_\f$ induced by $\f$ as in \eqref{bphi} satisfies $b_\f(e_1,e_1) = -2p_2^2f_1\,e^{1234567}$ and $b_\f(e_4,e_4) = -p_2 f_1 p_{12} e^{1234567}$. 
Since $b_\f$ is definite, we must have $p_2 p_{12}  f_1\neq0$. 

The generic derivation $D\in\Der(\frs_1)$ has the following expression with respect to the basis $(e_1,\ldots,e_7)$ of $\frs_1$ 
\[
D =  \left( 
\begin{array}{ccccccc} 
 h_{{1}}&h_{{2}}&h_{{3}}&0&0&h_{{4}}&h_{{5}}\\ \noalign{\medskip}
 0&\tfrac12\,h_{{1}}&h_{{6}}&0&0&h_{{2}}&0\\ \noalign{\medskip}
 0&-h_{{6}}&\tfrac12\,h_{{1}}&0&0&h_{{3}}&0\\ \noalign{\medskip}
 0&h_{{7}}&h_{{8}}&\tfrac12\,h_{{1}}&h_{{6}}&h_{{9}}&0\\ \noalign{\medskip}
 0&-h_{{8}}&h_{{7}}&-h_{{6}}&\tfrac12\,h_{{1}}&h_{{10}}&0\\ \noalign{\medskip}
 0&0&0&0&0&0&0\\ \noalign{\medskip}
 0&0&0&0&0&h_{{11}}&h_{{12}}
 \end {array} \right), 
\]
where $h_i\in\R$. 
Since $\f$ is a semi-algebraic soliton, there is some $\lambda\in\R$ such that the 3-form $D^*\varphi+\lambda \varphi$ on $\frs_1$ is exact. 
Under the constraint $p_2 p_{12}  f_1\neq0$, this implies that $\lambda=0$ and that $\frz(\frs_1) = \span_\R(e_1,e_7)\subset \ker(D)$. 
By Proposition \ref{signLambda}, we then have $|\tau|_\f=0$, i.e., the G$_2$-structure $\f$ is torsion-free. 
However, $\frs_1$ does not carry any torsion-free $\G_2$-structure by Corollary \ref{TFClass}. 
\end{proof}

\bigskip\noindent
{\bf Acknowledgements.}  
The authors were supported by GNSAGA of INdAM and by the project PRIN 2017  ``Real and Complex Manifolds: Topology, Geometry and Holomorphic Dynamics''. 
The authors would like to thank Giovanni Calvaruso for useful conversations and the anonymous referees for their valuable comments and suggestions.

\end{document}